\documentclass[10pt]{amsart}

\usepackage{filecontents}

\usepackage{amsmath}
\usepackage{amssymb}
\usepackage{amsthm}
\usepackage{amsfonts, dsfont}
\usepackage{paralist}
\usepackage{graphics} 
\usepackage{epsfig} 
\usepackage{graphicx}  
\usepackage{epstopdf}
\usepackage{epstopdf}
\usepackage{verbatim}
\usepackage{mathrsfs}
\usepackage{mathtools}
\usepackage{pstricks}
\usepackage{cleveref}
\usepackage{relsize}
\usepackage{tikz}
\usetikzlibrary{matrix}
\usepackage{subcaption}
\usepackage{pgfplots}
\usepackage{fixltx2e}
\usepackage{enumitem}
\usepackage{ upgreek }
\usepackage{bm}
\usepackage{appendix}
\usepackage{ulem}

\usepackage[abbrev]{amsrefs}

\usepackage{amssymb}

\usepackage[all,cmtip]{xy}

\usepackage{xcolor}

\definecolor{ao(english)}{rgb}{0.0, 0.5, 0.0}

\numberwithin{equation}{section}

\theoremstyle{plain} 
\newtheorem{thm}{Theorem}[section]
\newtheorem{cor}{Corollary}[section]
\newtheorem{lem}{lem}[section]
\newtheorem{prop}{Proposition}[section]

\theoremstyle{remark}
\newtheorem{rem}{Remark}[section]

\begin{document}

\title[Deep relaxation of SGD and singular perturbations]{Deep Relaxation of Controlled Stochastic Gradient Descent via Singular Perturbations}
\thanks{The first author is member of the Gruppo Nazionale per l'Analisi Matematica, la Probabilit\`a e le loro Applicazioni (GNAMPA) of the Istituto Nazionale di Alta Matematica (INdAM). He also  participates in the King Abdullah University of Science and Technology (KAUST) project CRG2021-4674 ``Mean-Field Games: models, theory, and computational aspects'', 
and in the project funded by the EuropeanUnion-NextGenerationEU under the National Recovery and Resilience Plan (NRRP), Mission 4 Component 2 Investment 1.1 - Call PRIN 2022 No. 104 of February 2, 2022 of Italian Ministry of University and Research; Project 2022W58BJ5 (subject area: PE - Physical Sciences and Engineering)  ``PDEs and optimal control methods in mean field games, population dynamics and multi-agent models".
\\
\indent
The second author was funded by the Deutsche Forschungsgemeinschaft (DFG, German Research Foundation) – Projektnummer 320021702/GRK2326 – Energy, Entropy, and Dissipative Dynamics (EDDy). Some of the results of this paper are part of his Ph.D. thesis \cite{kouhkouh22phd} which was conducted when he was a Ph.D. student at the University of Padova.
}

\author{Martino Bardi}
\address {Martino Bardi \newline 
{Department of Mathematics “T. Levi-Civita”, 
University of Padova, via Trieste, 63}, 
{I-35121 Padova, Italy}
}
\email{\texttt{bardi@math.unipd.it}}

\author{Hicham Kouhkouh}
\address{Hicham Kouhkouh \newline 
{RWTH Aachen University, Institut f\"ur Mathematik,  
RTG Energy, Entropy, and Dissipative Dynamics, 
Templergraben 55 (111810),}
{52062, Aachen, Germany} \newline 
\texttt{Current address: } 
Department of Mathematics and Scientific Computing, 
NAWI, University of Graz, 
8010, Graz, Austria 
}

\email{\texttt{kouhkouh@eddy.rwth-aachen.de}, \quad \texttt{hicham.kouhkouh@uni-graz.at}}

\date{\today}

\begin{abstract}
We consider a singularly perturbed system of stochastic differential equations proposed by Chaudhari et al. (Res. Math. Sci. 2018) to approximate the Entropic Gradient Descent in the optimization of deep neural networks, via homogenisation. We embed it in a much larger class of two-scale stochastic control problems and rely on convergence results for Hamilton-Jacobi-Bellman equations with unbounded data proved recently by ourselves ({ESAIM Control Optim. Calc. Var. 2023}). We show that the limit of the value functions is itself the value function of an effective control problem with extended controls, and that the trajectories of the perturbed system converge in a suitable sense to the trajectories of the limiting effective control system. These rigorous results improve the understanding of the convergence of the algorithms used by  Chaudhari et al., as well as of their possible extensions where some tuning parameters are modeled as dynamic controls. 
\end{abstract}



\maketitle


\section{Introduction}
This paper develops some stochastic control and Hamilton-Jacobi-Bellman methods related to general unconstrained non-convex optimization problems
\begin{equation}
\label{optimization problem}
    \min_{x\in\mathds{R}^{n}} \;\phi(x)
\end{equation}
where $\phi$ is a scalar function,  usually highly nonlinear. This very classical problem received recently a considerable amount of new attention in connection with its role in training deep neural networks. We do not enter here in the numerous practical applications of these problems and in the particular forms that the loss function $\phi$ may have,  and refer instead to the survey \cite{lecun2015deep}, the papers \cite{baldassi2015subdominant, chaudhari2018deep, chaudhari2019entropy, li2017stochastic, pittorino2021entropic} and the references therein. Let us just recall that in real-world applications the dimension $n$ of the unknown variable is extremely large, so that classical algorithms are often 
 not efficient enough because of the well-known curse of dimensionality, and the function $\phi$ may have poor smoothness properties.
 
Gradient descent in continuous version, given by
\begin{equation*}
   \dot X_{t} = -\nabla \phi(X_{t}) \,,
\end{equation*}
is known to converge to global minima for smooth convex functions $\phi$, otherwise to merely local minima or even saddles. The starting point of most modern algorithms is \textit{stochastic  gradient descent} (SGD)
\begin{equation*}
    \text{d}X_{t} = -\nabla \phi(X_{t})\text{d}t + \sigma \text{d}W_{t}
\end{equation*} 
where $W_{\cdot}$ is a Wiener process and the stochasticity arises from computing $\nabla \phi$ only on a lower dimensional set of randomly sampled directions, called \textit{mini-batch}. 

The papers  \cite{baldassi2015subdominant, chaudhari2019entropy} proposed to modify the loss function $\phi$ by an associated  local entropy $\phi_{\gamma}$, see \eqref{local entropy} below. This was shown to be effective in training high dimensional deep neural networks, see \cite{chaudhari2019entropy, chaudhari2018deep, pittorino2021entropic, pavon2022local}. The gradient of $\phi_{\gamma}$ is an average in space with respect to a Gibbs measure, see \eqref{gradient local entropy}, \eqref{Gibbs}, and such measure can be approximated by the long time average of an auxiliary stochastic differential equation \eqref{fast process}. Chaudhari, Oberman, Osher, Soatto, and Carlier \cite{chaudhari2018deep} show that the convergence of the resulting algorithm can be  better understood by means of the homogenisation theory for two-scale stochastic differential equation. A first goal of this paper is to put this singular perturbation result on a more general and solid mathematical ground.
We recall that, according to \cite{chaudhari2018deep}, ``Deep neural networks achieved remarkable success in a number of domains(...). A rigorous understanding of the roots of this success, however,  remains elusive".

A second goal of the paper is to add a control variable to the entropic-SGD system and study the singular perturbation in the framework of stochastic control. As a first example, we follow Li, Tai and E \cite{li2017stochastic} and consider the so-called learning rate of the algorithm as a control. However, our methods work for much more general control systems. Our main result does not require a gradient structure for the drift and allows a nonlinear dependence on a vectorial control. The rest of this Introduction gives more details on the problems, our results, and some related literature.

\subsection{{ Local entropy and deep relaxation}} 
\label{LEDR}
Following \cite{chaudhari2019entropy}
we consider $\phi_{\gamma}$, a regularization of the loss function $\phi$, such that
\begin{equation}
\label{local entropy}
    \phi_{\gamma}(x) \coloneqq -\frac{1}{\beta}\log\left(G_{\beta^{-1}\gamma}\ast\exp(-\beta \phi(x))\right)
\end{equation}
where
\begin{equation*}
    G_{\beta^{-1}\gamma}(x) \coloneqq (2\pi\gamma)^{-n/2}\exp\left(-\frac{\beta}{2\gamma}\,|x|^{2}\right)
\end{equation*}
is the heat kernel, and $\beta,\gamma>0$ are fixed parameter. The function $\phi_{\gamma}$ plays the role of a \underline{local entropy} and it is a smooth approximation of $\phi$ as $\gamma\to 0$. The parameter $\beta$ corresponds in physics to the inverse of the temperature and it will remain fixed in our analysis. Note that, by setting 
\begin{equation}
\label{eq: V in application}
    \Phi(y,x) := \phi(y) + \frac{1}{2\gamma}|x-y|^{2} ,
\end{equation}
we can write the local entropy as 
\[
    \phi_{\gamma}(x) = -\frac{1}{\beta}\log\left(\frac{1}{(2\pi\gamma)^{-n/2}} \int_{\mathds{R}^{n}} \exp\left(-\beta \Phi(y,x)\right)
    \,\text{d}y \right) .
\]
Then it is well-defined if $\Phi$ is a {confining potential} for all $x$, i.e., $\exp\left(-\beta \Phi(\cdot,x)\right)\in L^1(\mathds{R}^n)$ and $\Phi(y,x)\to +\infty$ as $|y|\to \infty$.
And this is true if $\phi$ has at most quadratic growth and $\gamma$ is chosen sufficiently small.

An additional  feature of this regularization procedure is that it ``provides a way of picking large, approximately flat,
regions of the landscape over sharp, narrow valleys, in spite of the latter
possibly having a lower loss'' \cite{chaudhari2019entropy}. This is important because in practice one is mostly interested in \underline{robust minima} lying in wide valleys. 
A further analysis and numerical validations of entropic gradient descent can be found in the recent paper
\cite{pittorino2021entropic}.

A direct computation shows that the gradient of the local entropy has the following nice structure \cite[Lemma 1]{chaudhari2018deep}
\begin{equation}
\label{gradient local entropy}
    \nabla \phi_{\gamma}(x)=\int_{\mathds{R}^{n}}\frac{x-y}{\gamma}\rho^{\infty}(\text{d}y;x)
\end{equation}
where, for a suitable  normalizing constant $Z(x)$,
\begin{equation}
\label{Gibbs}
\rho^{\infty}(y;x) \coloneqq 
\frac{1}{Z(x)} \exp\left(-\beta\left( \phi(y) + \frac{1}{2\gamma}|x-y|^{2}\right)\right) .
\end{equation}
Next we recognize that  $\rho^{\infty}$ is the density of the Gibbs invariant measure of the process
\begin{equation}
\label{fast process}
\text{d}Y_{s}  = -\nabla_{y} \Phi(Y_{s},x)\,\text{d}s + \sqrt{\frac{2}{\beta}}\,\text{d}W_{s},
\end{equation}
where $W_s$ is a $n$-dimensional Brownian motion and $x$ is frozen, provided  $\Phi$ is smooth enough. We will assume that $\phi\in C^1$ with a Lipschitz  gradient, then this fact 
can be found in \cite{bogachev2010invariant}, 
whereas for $\phi\in C^2$ one can check directly that $\mathcal L^* \rho^\infty =0$, where $\mathcal L$ is the generator of the process \eqref{fast process}.

Now the fact that $\nabla \phi_{\gamma}$ in \eqref{gradient local entropy} is the average of $y\mapsto \frac{1}{\gamma}(x-y)$ over a Gibbs measure is reminiscent of what often happens  in homogenization and in singular perturbations. 
Indeed, the authors of  \cite{chaudhari2018deep} introduce the following system of \textit{singularly perturbed SDEs} 
\begin{equation}
\label{homo_sys}
    \begin{aligned}
    \text{d}X_{s} & = -\nabla_{x}\Phi(Y_{s},X_{s})\,\text{d}s ,\quad X_{0}=x\in\mathds{R}^{n}\\ 
    \text{d}Y_{s} & = -\frac{1}{\varepsilon}\nabla_{y} \Phi(Y_{s},X_{s})\,\text{d}s + \sqrt{\frac{2}{\varepsilon\beta}}
    \,\text{d}W_{s}, \quad Y_{0}=y\in\mathds{R}^{n}.
    \end{aligned}
\end{equation}
They claim that such system converges, as $\varepsilon\to 0$, to 
\begin{equation*}
    \text{d} \hat{X}_{s} = \int_{\mathds{R}^{n}} -\frac{1}{\gamma}(X_{s}-y)\rho^{\infty} (\text{d}y;X_{s})\,\text{d}s,\quad \hat{X}_{0}=x\in\mathds{R}^{n}
\end{equation*}
which also writes, by \eqref{gradient local entropy}, 
\begin{equation}
\label{limit_homo-sys}
    \text{d} \hat{X}_{s} = -\nabla \phi_{\gamma}(\hat{X}_{s})\text{d}s,\quad \hat{X}_{0}=x\in\mathds{R}^{n} ,
\end{equation}
that is the gradient descent of the regularized loss function.

The result we get on this simple model, as a consequence of our more general analysis for controlled processes, is the following.

\begin{cor}
\label{cor:nocontrol}
Let $\phi\in C^{1}(\mathds{R}^{n})$ with $\nabla\phi$  Lipschitz continuous. Then, for all $T>0$, for  $\gamma$ in \eqref{eq: V in application} small enough, \\
(i)\, for any $y\in \mathds{R}^{n}$ the $x$-component  of the trajectory $(X^{\varepsilon}, Y^{\varepsilon})$ of \eqref{homo_sys} converges to the solution of \eqref{limit_homo-sys} in the sense
	\begin{equation*}
		\lim\limits_{\varepsilon\to 0} \left(\int_{0}^{T}\mathds{E}\left[ |X^{\varepsilon}_{s} - \hat{X}_{s}|^{2} \right] \text{d}s +\mathds{E}\left[|X^{\varepsilon}_{T} - \hat{X}_{T}|^{2}\right]\right)= 0,
	\end{equation*}\\
(ii)\, if for a sequence of processes $(X^{\varepsilon_{n}}, Y^{\varepsilon_{n}})$ solving \eqref{homo_sys} , with $\varepsilon_{n}\to 0$, there is a deterministic process $\bar{x}_{\cdot}$ such that
	\begin{equation*}
		\lim\limits_{\varepsilon_{n}\to 0} \int_{0} ^{T}\mathds{E}\left[ |X^{\varepsilon_{n}}_{s} - \bar{x}_{s}|^{p} \right]\text{d}s = 0, 
	\end{equation*}
	for some $p\in [1,2]$, then $\bar{x}_{\cdot}$ satisfies \eqref{limit_homo-sys}.
\end{cor}

The proof is postponed to Section  \ref{sec:DR}.

Note that the entropic gradient descent \eqref{limit_homo-sys} does not involve the gradient of the loss function $\phi$, which is usually hard to compute, because $\nabla \phi_\gamma$ depends  on $\phi$ only via the Gibbs  measure $\rho^\infty$. In practical algorithms such measure is  computed via long time averages  of the  process \eqref{fast process}, and in the discrete-time scheme $\nabla \phi$ is approximated by suitably chosen partial gradients, called \textit{mini-batches}.

 \subsection{Deep relaxation with control of the learning rate}
 \label{sec: app}

Following the model in \cite[\S 4]{li2017stochastic}, we introduce in \eqref{homo_sys} a control parameter $u_s$ playing the role of a Learning Rate. Choosing this rate optimally  allows to control in a dynamic way to what extent the process $X^\varepsilon_{\cdot}$ (and its limit $\hat{X}_{\cdot}$) should follow the gradient descent, in other words, how trustful the gradient descent direction is. Usually the control $u$ takes values in $[0,1]$. In the sequel we consider $U\subseteq\mathds{R}$ as a compact set of values that the control $u$ can take, and we write the system of \textit{singularly perturbed controlled SDEs}
\begin{equation}
\label{homo_sys_con}
    \begin{aligned}
    \text{d}X^\varepsilon_{s} & = -u_{s}\,\nabla_{x}\Phi(Y^\varepsilon_{s},X^\varepsilon_{s})\,\text{d}s + \sqrt{2}\sigma(X^\varepsilon_{s},Y^\varepsilon_{s},u_{s})\,\text{d}W_{s}, 
    \\
    \text{d}Y^\varepsilon_{s} & = -\frac{1}{\varepsilon}\nabla_{y} \Phi(Y^\varepsilon_{s},X^\varepsilon_{s})\,\text{d}s + \sqrt{\frac{2}{\varepsilon\beta}} 
    \,\text{d}W_{s}, 
    \end{aligned}
\end{equation}
where $\Phi$ is defined in \eqref{eq: V in application}, and $\sigma$ is a diffusion term that we add for the sake of generality and is allowed to be zero. The optimal learning rate should provide a balance between \textit{exploitation} (how fast at each step should we follow the drift) and  \textit{exploration} (how much at each step should we diffuse and look around). Given an appropriate payoff function for the problem of tuning the learning rate, we can write an optimal control problem of the form 
\begin{equation}
\label{J}
    \begin{aligned}
        \min\limits_{u} &\; \mathds{E}\left[g(X^\varepsilon_{T})e^{\lambda(t-T)} + \int_{t}^{T}\ell(s,X^\varepsilon_s,Y^\varepsilon_s,u_s)e^{\lambda(t-s)}\,\text{d}s\; \bigg|\; X^\varepsilon_{t}=x,\;Y^\varepsilon_{t}=y\right]
    \end{aligned}
\end{equation}
subject to \eqref{homo_sys_con}, where $\lambda$ is a non negative constant, and $g,\ell$ satisfy some growth assumptions that we will later make precise in section \ref{sec:OCP sys}, and can be chosen according to the performance we seek (e.g., minimizing $\mathds{E}[\phi(X^\varepsilon_{T})]$ or the expected distance of $X^\varepsilon$ from a reference trajectory). 

Our first main result, Theorem \ref{thm: value function sol limit pde}, says that the problem \eqref{homo_sys_con}-\eqref{J} converges in a variational sense (i.e., value function goes to value function) as $\varepsilon\to 0$, to a limit \underline{effective} control problem with \underline{extended controls} whose dynamics is 
\begin{equation}
\label{limit SGD}
    \text{d}\hat{X}_{s} = \overline f(\hat{X}_{s}, \upnu_{s})\,\text{d}s + \sqrt{2}\overline{\sigma}(\hat{X}_{s},\upnu_{s})\,\text{d}W_{s} \,, \quad \hat X_{0}=x\in\mathds{R}^{n} \,,
\end{equation}
\begin{equation}
\label{fbar}
\begin{aligned}
    & \overline{f}(\hat{x},\upnu) :=  - \int_{\mathds{R}^{n}} \frac{\hat x-y}{\gamma}\upnu(y)\rho^{\infty}(\text{d}y; \hat x) , \quad \upnu\in U^{ex}:= L^\infty(\mathds{R}^m, U),\\
    & \overline{\sigma}(\hat{x},\upnu) :=\sqrt{ \int_{\mathds{R}^{n}}\sigma\sigma^{\top}(\hat{x},y,\upnu(y)) \rho^{\infty} (\text{d}y;\hat x)},
\end{aligned}
\end{equation}
where $\sqrt M$ denotes the matrix square root of a positive semi-definite matrix $M$, and the payoff functional has the same form as \eqref{J} with the effective running payoff
\begin{equation}
\label{lbar}
	\overline{\ell}(s, \hat{x},\upnu) := \int_{\mathds{R}^{n}} \ell(s,\hat x,y,\upnu(y))\rho^{\infty}(\text{d}y; \hat x) .
   \end{equation}
Note that if $\upnu_t(y) = u_t$ is constant in $y$,  then $\overline f(x_{t}, \upnu_t)={ -}u_{t}\nabla \phi_{\gamma}(x_{t})$, where $\phi_{\gamma}$ is the local entropy associated to $\phi$, see \eqref{local entropy}, and its gradient is given by \eqref{gradient local entropy}. 
Therefore,  by taking $\sigma\equiv 0$ and $U=\{1\}$, we recover \eqref{limit_homo-sys} as a  particular case.

Let us illustrate such result on a practical example.
Let $\mathcal{U}$ be the set  of measurable functions $[0,T]\to U$. Consider the  value function of the deterministic control problem consisting of minimizing the loss function via trajectories of the  entropic gradient descent controlled by the learning rate, i.e., 
\begin{equation}\label{limit_sp-learn}
    \begin{aligned}
       \mathcal{V}(x) := & \inf\limits_{u_{\cdot}\in \mathcal{U}}\;  \phi(X_{T}) \\
        & \quad \text{s.t. }\; \dot X_{t} = -u_{t}\nabla \phi_{\gamma}(X_{t}) , \quad t\in [0,T], \quad  X_{0} = x \in \mathds{R}^{n}. 
    \end{aligned}
\end{equation}
Its \textit{deep relaxation} is the \textit{singularly perturbed} optimal control problem
\begin{equation}
\label{sp-learn}
    \begin{aligned}
        \mathcal{V}^{\varepsilon}(x,y) := & \inf\limits_{u_{\cdot}\in \mathcal{U}}\; \mathds{E}[\phi(X^{\varepsilon}_{T})]\\
        & \quad \text{s.t. }\; 
          \text{d}X^{\varepsilon}_{t} = - u_{t} \frac{X_{t}^{\varepsilon} - Y^{\varepsilon}_{t}}{\gamma} \,\text{d}t,\\        & \quad \quad\quad \,
        \text{d}Y^{\varepsilon}_{t} = -\frac{1}{\varepsilon}\left(\nabla \phi(Y^{\varepsilon}_{t}) - \frac{X_{t}^{\varepsilon} - Y^{\varepsilon}_{t}}{\gamma}\right)\,\text{d}t + \sqrt{\frac{2}{\varepsilon\beta}} \,\text{d}W_{t} , \\ 
        & \quad \quad \quad X^{\varepsilon}_{0} = x \in \mathds{R}^{n}, \; Y^{\varepsilon}_{0} = y \in \mathds{R}^{n},\quad t\in [0,T] ,
    \end{aligned}
\end{equation}
where $y$ is arbitrary.
\begin{thm}
\label{thm: practice}
Let $\phi\in C^{1}(\mathds{R}^{n})$ with  $\nabla\phi$ Lipschitz continuous.  Then for all $\gamma$ in \eqref{eq: V in application}  small enough
\begin{equation*}
    \lim\limits_{\varepsilon\to 0} \mathcal{V}^{\varepsilon}(x,y) \leq \mathcal{V}(x)
\end{equation*}
locally uniformly in $x,y\in \mathds{R}^n$, i.e., the perturbed dynamics yields a  value not larger than the one with a controlled full gradient descent.
\end{thm}
The proof is based on the fact that $ \lim_{\varepsilon\to 0} \mathcal{V}^{\varepsilon}$ is the following value function:
\begin{equation}\label{effective_sp-learn}
    \begin{aligned}
        \overline{\mathcal{V}}(x) := & \inf
      \limits_{\upnu_{\cdot}\in \mathcal{U}^{ex}}\; 
       \phi(\hat X_{T})    \quad \text{s.t. }\; \frac{d}{ds}\hat{X}_{s} = 
    \overline f(\hat{X}_{s}, \upnu_{s}) , \quad t\in [0,T] 
         \quad  X_{0} = x \in \mathds{R}^{n} ,
    \end{aligned}
\end{equation}
where $\mathcal{U}^{ex}$ is the set of measurable functions $[0,T]\to U^{ex}$  and $\overline f$ is  the effective dynamics \eqref{fbar}. The details are postponed to Section \ref{sec:DR}.

The other main results of the paper, Theorems \ref{conv-thm-1} and \ref{conv-thm 2}, imply  that for $\sigma \equiv 0$ also the trajectories of \eqref{homo_sys_con} converge  to those of \eqref {limit SGD}  in a sense similar to the preceding Corollary \ref{cor:nocontrol}. A consequence of these properties is the possibility of approximating optimal trajectories of the effective problem \eqref{effective_sp-learn} by sub-optimal trajectories for the perturbed problem \eqref{sp-learn}. By this we mean pairs $(X^{\varepsilon_{n}}, Y^{\varepsilon_{n}})$ such that
\begin{equation}\label{subop}
  \mathds{E}[\phi(X^{\varepsilon_n}_{T})]   \leq    \mathcal{V}^{\varepsilon_n}(x,y) + o(1) , \quad \text{ as } \varepsilon_n\to 0 .
\end{equation}

\begin{cor}
\label{cor:trajectories (old)} 
Under the assumptions of Theorem \ref{thm: practice} with $U$  convex, let $(X^{\varepsilon_{n}}\!,\!Y^{\varepsilon_{n}}\!)$  be  trajectories of \eqref{sp-learn} satisfying \eqref{subop}, 
	\begin{equation*}
		\lim\limits_{\varepsilon_{n}\to 0} \int_{0} ^{T}\mathds{E}\left[ |X^{\varepsilon_{n}}_{s} - \bar{x}_{s}|^{p} \right] \text{d}s= 0, 
	\end{equation*}
	for a deterministic process $\bar{x}_{\cdot}$ 
 and some $p\in [1,2]$, and 
 	\begin{equation}\label{ineq: cor}
			\limsup \limits_{\varepsilon_{n}\to 0} \mathds{E}\left[ \phi(X^{\varepsilon_{n}}_{T}) \right] \geq  \phi (\bar{x}_{T}) .
	\end{equation}
	Then $\bar{x}_{\cdot}$  is optimal for the problem \eqref{effective_sp-learn}, i.e., for some $\upnu_.\in \mathcal{U}^{ex}$ it is a trajectory of the effective system in \eqref{effective_sp-learn} and $\phi (\bar{x}_{T}) = \overline{\mathcal{V}}(x)$. 
\end{cor}
The proof can be found in Section \ref{sec:DR}.

\subsection{General singular perturbations in stochastic control}
\label{sec: intro}
We are going to embed the deep relaxation problems described so far in a more general setting of singularly perturbed stochastic control systems, namely, 
\begin{equation}
\label{dynamics}
\begin{aligned}
    \text{d}X_{s} &= f(X_{s},Y_{s},u_{s})\,\text{d}s + \sqrt{2}\,\sigma^{\varepsilon}(X_{s},Y_{s},u_{s})\,\text{d}W_{s},\\
    \text{d}Y_{s} & = \frac{1}{\varepsilon}\,b(X_{s},Y_{s})\,\text{d}s + \sqrt{\frac{2}{\varepsilon}\,}\varrho(X_{s},Y_{s})\,\text{d}W_{s},
\end{aligned}
\end{equation}
where $X_{s}\in \mathds{R}^n$ is the \textit{slow} dynamics, $Y_{s}\in \mathds{R}^m$ is the \textit{fast} dynamics, $u_{s}$ is the control taking values in a given compact set $U$, and $W_{s}$ is a multidimensional Brownian motion. 
We will make suitable regularity and growth assumptions on the data of such system, but we will not require the gradient structure of the drifts in \eqref{homo_sys_con}. 
The diffusion coefficient of the process $X_.$ can be degenerate (i.e. $\sigma^{\varepsilon}=0$ is allowed), whereas the one of $Y_.$ is required to be non-degenerate. The precise assumptions are given in Section \ref{sec: setting}.
We will consider optimization problems with cost/payoff function of the general form 
\begin{equation}
\label{cost function}
  \begin{aligned}
    &J(t,x,y,u):=\\
    &\mathds{E}\left[ e^{\lambda(t-T)}g(X_{T},Y_{T}) + \int_{t}^{T}\ell(s,X_{s},Y_{s},u_{s})e^{\lambda(t-s)}\text{d}s \; \bigg|\; X_{t}=x,\;Y_{t}=y \right].
    \end{aligned}
\end{equation} 
The value function $V^{\varepsilon}(t,x,y)$ of such problem is known to solve in the viscosity sense a fully nonlinear, degenerate parabolic PDE of Hamilton-Jacobi-Bellman type in $(0,T)\times\mathds{R}^n\times \mathds{R}^m$. From our results in the companion paper \cite{bardi2023singular} we have the convergence of $V^{\varepsilon}$ to the solution $V(t,x)$ of a suitable HJB PDE in $(0,T)\times\mathds{R}^n$, as $\varepsilon\to 0$. Here we first show that $V$ is the value function of a limit effective optimal control problem in reduced dimension $n$,  but with a larger set of {\it extended controls} $U^{ex}$. This limit problem is based on averaging with respect to a probability measure that is fully characterized as invariant measure of an ergodic process, although in general it is not as explicit as the Gibbs measure in \eqref{limit SGD}, and \eqref{lbar}. Then we make suitable choices of  $\ell$ and $g$ in the functional $J$ \eqref{cost function} to deduce two convergence theorems for trajectories of \eqref{dynamics} to trajectories of the effective system. These results are new also for problems with data bounded in the fast variables $Y$, for which the convergence in HJB equations was already proved in \cite{bardi2011optimal,bardi2010convergence}.

The generality of the system \eqref{dynamics} is motivated by applications different from the Deep Learning problems studied here. The results of the present paper can be used for models of pricing and trading derivative securities in financial markets with stochastic volatility, as  in \cite{bardi2010convergence} {and \cite{fouque2011multiscale}}, or for applications in economics and advertising theory as  in \cite{bardi2011optimal}.

There is a wide literature on singular perturbations for control systems that goes back to the late 60's \cite{kokotovic1999singular, kushner2012weak}, and also 
for diffusion processes, with and without control, and many different models with fast variables have been studied since then, both in deterministic and stochastic settings and using methods of probability, analysis, measure theory, or control. We refer the reader to {the introductions of \cite{bardi2011optimal,bardi2010convergence} and the large but non-exhaustive list of references therein.  Let us mention some additional papers on singular perturbation problems for stochastic differential equations: 
for the case without control the papers \cite{pardoux2001poisson,pardoux2003poisson,pardoux2005poisson} by Pardoux and Veretennikov; for problems with control the work of Borkar and Gaitsgory \cite{borkar2007averaging, borkar2007singular} where the control acts on both the slow and  the fast variables, and is analyzed by means of Limit Occupational Measures. 

The paper is organized as follows. In Section \ref{sec: setting} we list the basic assumptions on \eqref{dynamics} and \eqref{cost function}, then recall our results from \cite{bardi2023singular} on the HJB equations and the convergence of their solutions that are crucial for our subsequent analysis.  In Section \ref{sec: conv traj} we study a new optimal control problem, the effective problem, and show that the limit $V$ of the value functions $V^\varepsilon$ is indeed the value function of such problem.
In Section \ref{sec: conv traj 2}
we study the convergence of the trajectories, and finally apply all this to the deep relaxation problems described in Sections \ref{LEDR} and \ref{sec: app}.

\section{The two scale stochastic control problem}
\label{sec: setting}

\subsection{The system}
\label{sec:stochastic system}
Let $(\Omega,\mathcal{F},\mathcal{F}_{t},\mathds{P})$ be a complete filtered probability space and let $(W_{t})_{t}$ be an $\mathcal{F}_{t}$-adapted standard $r$-dimensional Brownian motion. We consider the stochastic control system \eqref{dynamics}, where $u_t$ is in the set of admissible control functions  $ \mathcal{U}$,  i.e., the set of $\mathcal{F}_{t}$-progressively measurable processes taking values in $U$.
 Note that it is much more general than the systems \eqref{homo_sys} and \eqref{homo_sys_con}
 arising in deep relaxation, in particular the vector fields do not have a gradient structure. However,  we will not present the results in the full generality of our paper \cite{bardi2023singular}, instead we strengthen a bit some conditions in order to simplify the statements of this section. In Section \ref{sec: conv traj} we will strengthen further the assumptions to get detailed results on the effective control problem.  We denote by $\mathds{M}^{n,m}$ (respec. $\mathds{S}^{n}$) the set of matrices of $n$ rows and $m$ columns (respec. the subset of $n$-dimensional squared symmetric matrices).

\noindent{\textit{Assumptions \textbf{(A)}:}}

\begin{enumerate}[label=\textbf{(A\arabic*)}]
	\item For a given compact set $U$, $f:\mathds{R}^{n}\times\mathds{R}^{m}\times U \rightarrow \mathds{R}^{n}$, 
$\sigma^{\varepsilon}: \mathds{R}^{n}\times\mathds{R}^{m}\times U \rightarrow \mathds{M}^{n,r}$, and
$b:\mathds{R}^{n}\times\mathds{R}^{m}\rightarrow \mathds{R}^{m}$  are continuous functions, Lipschitz continuous in $(x,y)$ uniformly with respect to $u\in U$ and $\varepsilon>0$. 
	\item The diffusion $\sigma^{\varepsilon}$ driving the slow variables $X_{t}$ satisfies
\begin{equation*}
    \lim\limits_{\varepsilon\rightarrow0}\sigma^{\varepsilon}(x,y,u) = \sigma(x,y,u)\quad\text{locally uniformly},
\end{equation*}
where $\sigma:\mathds{R}^{n}\times\mathds{R}^{m}\times U \rightarrow \mathds{M}^{n,r}$ satisfies the same conditions as $\sigma^{\varepsilon}$. 
	\item The diffusion $\varrho$  is constant such that $\varrho\varrho^{\top}=\Bar{\varrho\,}\mathds{I}_{m}$ where $\Bar{\varrho}>0$ is a constant and $\mathds{I}_{m}$ is the identity matrix.
	\item The drift $b$ satisfies the \textit{strong monotonicity condition}
    \begin{equation}
        \exists\,\kappa>0 \text{ s.t. } \,(b(x,y_{1})-b(x,y_{2}))\cdot(y_{1}-y_{2}) \leq -\kappa \, |y_{1}-y_{2}|^{2},\quad \forall\, x,y_{1},y_{2}.
    \end{equation}
\end{enumerate}
Note that assumption \textbf{(A1)} implies that $f, \sigma^{\varepsilon},$ and $b$ have linear growth in both $x$ and $y$, that is, for some positive constant $C$,
\begin{equation}
\label{assumption-slow}
    |f(x,y,u)|,\|\sigma^{\varepsilon}(x,y,u)\|\leq C(1+|x|+|y|),\quad \forall\; x,y, u,\;\forall \varepsilon>0,
\end{equation}
\begin{equation}
\label{assumption-fast}
    |b(x,y| \leq C(1+|x|+|y|),\quad \forall\; x,y .
\end{equation}
We remark that we will not make any non-degeneracy assumption on the matrices $\sigma^{\varepsilon},\sigma$, so the cases $\sigma^{\varepsilon},\sigma\equiv0$ are allowed.

\subsection{The optimal control problem and HJB equation}
\label{sec:OCP sys}
The functional $J$ with finite horizon on the time interval $[t, T]$, $t>0$, and  discount factor $\lambda\geq 0$, is defined by \eqref{cost function}. In this section, for consistency with our paper \cite{bardi2023singular}, we consider $J$ as a payoff  and the optimal control problem consists of maximizing it. The minimization problem of the previous section is easily recovered by changing signs to $\ell$ and $g$. Then the value function is 
\begin{equation}
    \label{value function}
    \tag{$\;OCP(\varepsilon)\;$}
    V^{\varepsilon}(t,x,y) := \sup\limits_{u\in\mathcal{U}}J(t,x,y,u),\quad \text{subject to }\; \eqref{dynamics}\,.
\end{equation}
The following assumption concerns the utility function $g$ and the running payoff $\ell$:

\noindent{\textit{Assumption \textbf{(B)}:}}  The discount factor is $\lambda\geq 0$, and the utility function $g:\mathds{R}^{n}\times\mathds{R}^{m}\rightarrow\mathds{R}$ and running payoff $\ell:[0,T]\times\mathds{R}^{n}\times\mathds{R}^{m}\times U\rightarrow\mathds{R}$ are continuous, Lipschitz  in $y$ uniformly in their other arguments, and satisfy
\begin{equation}
\label{assumption-cost}
    \exists\;K>0\;\text{ s.t. }\; |g(x,y)|,|\ell(s,x,y,u)|\leq K(1+|x|^{2} + |y|
    ),\;\forall s\in[0,T], x,y, u .
\end{equation}

The HJB equation associated via Dynamic Programming to the value function $V^{\varepsilon}$ is
\begin{equation}
    \label{HJB}
    -V^{\varepsilon}_{t} + F^{\varepsilon}\left(t,x,y,V^{\varepsilon},D_{x}V^{\varepsilon},\frac{D_{y}V^{\varepsilon}}{\varepsilon}, D^{2}_{xx}V^{\varepsilon},\frac{D^{2}_{yy}V^{\varepsilon}}{\varepsilon}, \frac{D^{2}_{x,y}V^{\varepsilon}}{\sqrt{\varepsilon}}\right) = 0 ,
\end{equation}
in $(0,T)\times\mathds{R}^{n}\times\mathds{R}^{m}$, complemented with the  terminal condition
\begin{equation}
\label{terminal cond}
    V^{\varepsilon}(T,x,y)=g(x,y) .
\end{equation} 
The Hamiltonian $F^{\varepsilon}:[0,T]\times\mathds{R}^{n}\times\mathds{R}^{m}\times\mathds{R}\times\mathds{R}^{n}\times\mathds{R}^{m}\times\mathds{S}^{n}\times\mathds{S}^{m}\times\mathds{M}^{n,m}\rightarrow\mathds{R} \;$ is 
\begin{equation*}
    F^{\varepsilon}(t,x,y,r,p,q,M,N,Z) := H^{\varepsilon}(t,x,y,p,M,Z) - \mathcal{L}(x,y,q,N)+\lambda r,
\end{equation*}
where
\begin{equation*}
    H^{\varepsilon}(t,x,y,p,M,Z) := \min\limits_{u\in U}\left\{ -\text{trace}(\sigma^{\varepsilon}\sigma^{\varepsilon\top}M) - f\cdot p - 2\text{trace}(\sigma^{\varepsilon}\varrho^{\top}Z^{\top}) - \ell  \right\}
\end{equation*}
where $\sigma^{\varepsilon},f$  are computed at $(x,y,u)$, $\ell=\ell(t,x,y,u)$,  and 
\begin{equation*}
    \mathcal{L}(x,y,q,N) := b(x,y)\cdot q + \text{trace}(\varrho \varrho^{\top}    N).
\end{equation*}
This is a fully nonlinear degenerate parabolic equation (strictly parabolic in the $y$ variables by the assumption (A3)). 
We define also the Hamiltonian $H$ as $H^{\varepsilon}$ when $\sigma^{\varepsilon}$ is replaced by $\sigma$:
\begin{equation}
    \label{H}
    H(t,x,y,p,M,Z) := \min\limits_{u\in U}\left\{ -\text{trace}(\sigma\sigma^{\top}M) - f\cdot p - 2\text{trace}(\sigma\varrho^{\top}Z^{\top}) - \ell  \right\} .
\end{equation}
The next  result is standard, see, e.g., \cite[Proposition 3.1]{bardi2010convergence} or in \cite[Prop. 2.1]{bardi2011optimal}. 

\begin{prop}     
\label{prelimit sol}
Assume (A) and (B). 
Then for any $\varepsilon>0$, the function $V^{\varepsilon}$ in \eqref{value function} is the unique continuous viscosity solution to the Cauchy problem \eqref{HJB}-\eqref{terminal cond} with at most quadratic growth in $x$ and $y$, i.e.,
\begin{equation*}
    \exists\;K>0\;\text{ such that }\; |V^{\varepsilon}(t,x,y)|\leq K(1+|x|^{2}+|y|^{2}),\quad \forall\;t\in [0,T],\;x\in\mathds{R}^{n},\;y\in\mathds{R}^{m}.
\end{equation*}
Moreover the functions $V^{\varepsilon}$ are locally equibounded.
\end{prop}
We remark that $V^{\varepsilon}$ is not bounded in $y$, contrary to what occurs in \cite{bardi2010convergence,bardi2011optimal}, but it has at most quadratic growth. This comes from the assumptions \eqref{assumption-slow} and  \eqref{assumption-cost}. 

\subsection{The limit PDE  and convergence of the value functions} 
Consider the \textit{fast subsystem},
\begin{equation}
    \label{fast subsys}
    \text{d}Y_{t}= b(x,Y_{t})\,\text{d}t + \sqrt{2}{\Bar{\varrho} }
    \,\text{d}W_{t},\quad Y_{0}=y\in\mathds{R}^m
\end{equation}
obtained by putting $\varepsilon=1$ in \eqref{dynamics}, freezing $x\in\mathds{R}^{n}$, {and applying (A3)}. It is a non-degenerate diffusion process that is known to be \textit{ergodic} by the monotonicity condition (A4), see \cite[\S 3]{bardi2023singular}.
\begin{prop}  
\label{prop:innv} 
Under the assumptions (A1), (A3), (A4), for all $x\in\mathds{R}^{n}$ the process \eqref{fast subsys} has a unique invariant probability measure $\mu_x$, which has finite moments of any order. Moreover, $x\mapsto \mu_x$ is Lipschitz with respect to the 2-Wasserstein distance.
\end{prop}   

\begin{proof}
For existence and uniqueness, we refer to \cite[\S 3.1]{bardi2023singular} and the references therein. See also \cite[Examples 5.1 \& 5.5]{bogachev2002uniqueness}. The moments being finite is \cite[Lemma 3.1]{bardi2023singular}. For Lipschitz continuity, see \eqref{bogachev}-\eqref{eq: lip W2} below or \cite[Lemma 4.1 \& eq. (4.9)]{bardi2023singular}.
\end{proof}

For $1\leq p <+\infty$, let $\mathfrak{m}_{p}(x)$ denote the $p$-moment of the invariant probability measure $\mu_{x}$ \vspace*{-0.5em}
\begin{equation}\label{def moment}
    \mathfrak{m}_{p}(x) := \int_{\mathds{R}^{m}} |y|^{p}\,\text{d}  \mu_{x}(y).
\end{equation} 

\begin{prop}
\label{prop:moment} 
Assume  (A1), (A3) and (A4) hold. 
Let $\mu_{x_{1}}, \mu_{x_{2}}$ be the unique invariant probability measures corresponding to \eqref{fast subsys} with $x_{1},x_{2}\in \mathds{R}^{n}$. Then 
\begin{equation*}
	\left| \mathfrak{m}_{p}^{1/p}(x_{1}) - \mathfrak{m}_{p}^{1/p}(x_{2}) \right| \leq \mathcal{W}_{p}(\mu_{x_{1}},\mu_{x_{2}}) .
\end{equation*}
In particular,
 there exists a constant $C>0$ such that
\begin{equation*}
	\left| \mathfrak{m}_{1}(x_{1}) - \mathfrak{m}_{1}(x_{2}) \right|  + \left| \mathfrak{m}_{2}^{1/2}(x_{1}) - \mathfrak{m}_{2}^{1/2}(x_{2}) \right| \leq C\, |x_{1} - x_{2}| ,
\end{equation*}
and for all $x\in\mathds{R}^{n}$ 
\begin{equation}\label{eq: moment growth}
	\mathfrak{m}_{1}(x) + \mathfrak{m}_{2}^{1/2}(x) \leq C(1 + |x|).
\end{equation}
\end{prop}

\begin{proof}
It is known that any measure $\pi\in \mathcal{P}(\mathds{R}^{m}\times \mathds{R}^{m})$ whose marginals are some measure $\mu\in \mathcal{P}(\mathds{R}^{m})$ and $\delta_{a}$ a Dirac measure supported in some $a\in\mathds{R}^{m}$, is such that $\pi = \mu \otimes \delta_{a}$. That is, such subset of $\mathcal{P}(\mathds{R}^{m}\times \mathds{R}^{m})$  is the singleton $\{\mu \otimes \delta_{a}\}$. Hence
\begin{equation*}
    \mathcal{W}_{p}^{p}(\mu,\delta_{a}) = \int_{\mathds{R}^{m}} |y-a|^{p} \, \text{d}  \mu(y).
\end{equation*}
So when $a=0$ we have $\; \mathcal{W}_{p}^{p}(\mu_{x},\delta_{0})  = \mathfrak{m}_{p}(x). \;$ 
Using the triangle inequality, one gets
\begin{equation*}
\begin{aligned}
	\mathfrak{m}_{p}^{1/p}(x_{1}) \leq \mathcal{W}_{p}(\mu_{x_{1}},\mu_{x_{2}}) + \mathcal{W}_{p}(\mu_{x_{2}},\delta_{0})
	 = \mathcal{W}_{p}(\mu_{x_{1}},\mu_{x_{2}}) + \mathfrak{m}_{p}^{1/p}(x_{2})
\end{aligned}
\end{equation*}
and then $\; \mathfrak{m}_{p}^{1/p}(x_{1}) - \mathfrak{m}_{p}^{1/p}(x_{2}) \leq  \mathcal{W}_{p}(\mu_{x_{1}},\mu_{x_{2}})$. \; Exchanging the roles of $x_{1},x_{2}$ yields the desired result. 

To prove the second statement, we recall from \cite[Corollary 1]{bogachev2014kantorovich} the estimate 
\begin{equation}\label{bogachev}
    \mathcal{W}_{2}(\mu_{x_{1}},\mu_{x_{2}}) \leq C\, \left( \int_{\mathds{R}^{m}} \big| b(x_{1},y) - b(x_{2},y) \big|^{2} \text{d}\, \mu_{x_{2}}(y) \right)^{\frac{1}{2}}.
\end{equation}
Then, with the Lipschitz continuity of $b$, one gets
\begin{equation}\label{eq: lip W2}
    \mathcal{W}_{2}(\mu_{x_{1}},\mu_{x_{2}}) \leq C|x_{1} - x_{2}|.
\end{equation}
Recalling $\mathcal{W}_{1}(\mu_{x_{1}},\mu_{x_{2}}) \leq \mathcal{W}_{2}(\mu_{x_{1}},\mu_{x_{2}})$, one has
\begin{equation*}
	\text{for } \, p=1,2\, :\quad \left| \mathfrak{m}_{p}^{1/p}(x_{1}) - \mathfrak{m}_{p}^{1/p}(x_{2}) \right| \leq C |x_{1} - x_{2}|.
\end{equation*}
Now let $x_{1}=x$, $x_{2} = 0$ and $p=1,2$. One has
\begin{equation*}
\begin{aligned}
	\mathfrak{m}_{1}(x)  \leq C|x| + \mathfrak{m}_{1}(0) \; \text{ and } \; \mathfrak{m}_{2}^{1/2}(x)  \leq C|x|  + \mathfrak{m}_{2}^{1/2}(0). 
\end{aligned}
\end{equation*}
Then $\; \mathfrak{m}_{1}(x) + \mathfrak{m}_{2}^{1/2}(x) \leq C(1 + |x|) \;$ for some $C>0$ constant depending only on  $b(0,\cdot\,),\bar{\varrho}$ in \eqref{fast subsys}.
\end{proof}

Now we can define the \textit{effective Hamiltonian}   $\overline{H}$ and  \textit{effective initial data} $\overline{g}$ associated to the singularly perturbed PDE \eqref{HJB} as follows
\begin{equation}
    \label{eff-Hamiltonian}
\overline{H}(t,x,p,P) := \int_{\mathds{R}^{m}} H(t,x,y,p,P,0) d\mu_{x}(y) , \quad \overline{g}(x) := \int_{\mathds{R}^m}g(x,y)\text{d}\mu_{{x}}(y) .
\end{equation}

Note that they are finite for all entries by the growth conditions \eqref{assumption-slow} and \eqref{assumption-cost}. We expect that the value function $V^{\varepsilon}(t,x,y)$, solution to \eqref{HJB}-\eqref{terminal cond}, converges,  as $\varepsilon\rightarrow0$, to a function $V(t,x)$ independent of $y$ which is the unique solution of the Cauchy problem
\begin{equation}
\label{CP-limit HJB}
\left\{\;
\begin{aligned}
    -V_{t} + \overline{H}(t,x,D_{x}V,D^{2}_{xx}V) + \lambda V(x) & = 0, & \text{in }\; (0,T)\times \mathds{R}^{n} ,\\
    \quad V(T,x) & = \overline{g}(x),& \text{in }\; \mathds{R}^{n} .
\end{aligned}
\right.
\end{equation}
This is the statement of the next theorem for which we need the following

\noindent{\textit{Assumption \textbf{(C)}:}} For $\Sigma :=\sigma\sigma^{\top}(x,y,u)$, one of these two conditions is satisfied
\begin{enumerate}[label=(\alph*)]
	\item $\Sigma$ is independent of $y$ and $u$, i.e. $\sigma=\sigma(x)$; or
	\item the drift of the fast process is independent of $x$, i.e. $b=b(y)$, and the matrix function  $x\mapsto \Sigma(x,y,u)$ has two continuous spatial derivatives satisfying
	\begin{enumerate}[label=(\roman*)]
		\item there exists $\theta(\cdot):\mathds{R}^{m}\to \mathds{R}$ such that 
		$\forall\, u\in U$, $x\in \mathds{R}^{n}$, and $y\in \mathds{R}^{m}$
		\begin{equation*}
		    \max\limits_{1\leq i \leq n} \left| \frac{\partial}{\partial x_{i}}\Sigma(x,y,u)\xi \cdot \xi \right| \leq \theta(y)|\xi|^{2},\; \xi\in \mathds{R}^{n}, \quad  \int_{\mathds{R}^{m}} |\theta(y)|\,\text{d} \mu(y) < +\infty ,
		\end{equation*}
		where $\mu$ in the unique invariant probability measure of the fast process \eqref{fast subsys} (that is now independent of $x$),
		\item $\exists\, K>0$ constant such that for all $u\in U$, $x\in \mathds{R}^{n}$, and $y\in \mathds{R}^{m}$
		\begin{equation*}
		    \max\limits_{1\leq i \leq n} \left| \frac{\partial^{2}}{\partial x_{i}^{2}}\Sigma(x,y,u)\xi \cdot \xi \right| \leq K\,|\xi|^{2},\quad \forall \xi\in \mathds{R}^{n}.
		\end{equation*}
	\end{enumerate}
\end{enumerate}
Note that in the  case (b),  thanks to Proposition \ref{prop:innv}, it is sufficient to have $\theta$ with at most a polynomial growth to satisfy the integrability assumption.

\begin{thm}
\label{thm-conv-value}
Assume (A), (B), and (C). Then the solution $V^{\varepsilon}$ to \eqref{HJB} converges uniformly on compact subsets of  ${ [0,T)}\times\mathds{R}^{n}\times\mathds{R}^{m}$ to the unique continuous viscosity solution of the limit problem \eqref{CP-limit HJB} satisfying a quadratic growth condition in $x$, i.e.
\begin{equation}
\label{qg}
    \exists\;K>0\;\text{such that }\; |V(t,x)|\leq K(1+|x|^{2}),\quad\forall\;(t,x)\in[0,T]\times\mathds{R}^{n} .
\end{equation}
Moreover, the convergence is uniform up to time $t=T$ if $g$ is independent of $y$, i.e., $\overline{g}(x)=g(x)$.
\end{thm}

\begin{proof}
See \cite[Theorem 4.4]{bardi2023singular}.
\end{proof}

In the next sections we provide an optimal control representation of the solution to this limit problem and prove the convergence of the trajectories of \eqref{dynamics} to a suitable limit system.

\section{The effective control problem}
\label{sec: conv traj}

The previous result states the convergence of the value functions $V^{\varepsilon}$ to the solution $V$  of a Cauchy problem of an effective PDE.  In this section we first show that such PDE is the HJB equation for a suitable optimal control problem, that we call {\it effective}. Next we show, under slightly stronger assumptions, that $V$ is indeed the value function of such problem. Since the conditions on the payoff/cost functional are very general, this can be interpreted as a weak ``variational" convergence result for the trajectories of \eqref{dynamics} to the trajectories of the effective system.  Such convergence will be made more precise in the final Section \ref{sec: conv traj 2}.

\subsection{A control interpretation of the limit PDE}\label{sec: control interpretation}

The following result is \cite[Proposition 3.6]{bardi2023singular}, and allows to represent the effective Hamiltonian \eqref{eff-Hamiltonian} as a Bellman Hamiltonian associated to an \textit{effective optimal control problem}, where the control set $U$ is replaced by  the \textit{extended control set} $\; U^{ex}:=\{ \nu : \mathds{R}^{m} \to U\, \text{measurable} \}. $
\begin{prop}
\label{representation-prop}
Assume (A) and (B). 
 Then the effective Hamiltonian \eqref{eff-Hamiltonian} satisfies 
\begin{equation}
    \label{representation}
    \overline{H}(t,x,p,P) =  \min\limits_{\nu \in U^{ex}}
\; \int_{\mathds{R} ^{m}} \left[    -\text{trace}(\sigma\sigma^{\top}P) - f\cdot p - \ell \right]\;\text{d}\mu_{x}(y)
\end{equation}where $\sigma$ and $f$ are computed in $(x,y,\nu(y))$, and $\ell$ in $(t,x,y,\nu(y))$. 
\end{prop}

\begin{rem}\label{rmk: extended control set}
The extended control set $U^{ex}$ contains a copy of $U$, the constant functions, and it coincides with $L^{p}((\mathds{R}^{m},\mu_{x}),U)$ for all $x$ and all $p\in[1, +\infty]$, since $\mu_{x}$ is finite and $U$ is compact.
\end{rem}

\begin{rem}
\label{rem: decoup1}
When the controls are decoupled from the fast variables, i.e.,
\begin{equation*}
\begin{aligned}
	f(x,y,u) = f_{1}(x,u) + & f_{2}(x,y) , \quad
	\sigma(x,y,u) = \sigma_{1}(x,u) + \sigma_{2}(x,y) ,\\
	\ell(x,y,u) &= \ell_{1}(x,u)+ \ell_{2}(x,y) ,
\end{aligned}
\end{equation*}
and $\sigma_{1}(x,u)^{\top}\sigma_{2}(x,y)=0$ for all $x,y,u$, then repeating the same proof of Proposition \ref{representation-prop} shows that extended controls are not necessary because \eqref{representation} holds with $\nu(y)=y$ and $U^{ex}$ replaced by $U$. This was done in \cite{kushner2012weak} and \cite{bardi2011optimal}, but the decoupling of $f$ does not occur in our motivating model \eqref{sp-learn}. 
Note that the condition $\sigma_{1}(x,u)^{\top}\sigma _{2}(x,y)=0$ is satisfied if the diffusion term $\sigma dW_{s}$ is of the form $\sigma_{1}dW^{1}_{s} + \sigma_{2}dW^{2}_{s}$ with $W^{1}$ and $W^{2}$ independent.
\end{rem}

\subsection{The limit $V$ is a value function}\label{sec:limit V} 
In view of the last result it is natural to define the effective drift and diffusion as 
\begin{equation}
\label{F,G}
    \overline f (x,\upnu) := \int_{\mathds{R}^{m}} f(x,y,\upnu(y))\text{d}\mu_{x}(y) \,, 
     \quad { \overline \sigma}(x,\upnu) := \sqrt{\int_{\mathds{R}^{m}}\sigma\sigma^{\top}(x,y,\upnu(y))\text{d}\mu_{x}(y)} \,,
\end{equation}
the measure $\mu_{x}$ being defined in  Proposition \ref{prop:innv},
and consider as  effective control system 
\begin{equation}
    \label{eff-dynamics}
    \left\{
    \begin{aligned}
        d\hat{X}_{s} &= \overline f(\hat{X}_{s},\upnu_{s})
        \text{d}s +  \sqrt{2}\overline \sigma(\hat{X}_{s},\upnu_{s})\;\text{d}W_{s} , \\
        \upnu_{\cdot} &\in \mathcal{U}^{ex},\quad \text{and }\; {\hat{X}_{t} =x}\in\mathds{R}^{n} ,
    \end{aligned}
    \right.
\end{equation}
where $\;\mathcal{U}^{ex}$ is the set of progressively measurable processes taking values in the extended control set $U^{ex}$. 
Define also the effective data for the payoff functional as
\begin{equation}\label{eff payoff data}
    \begin{aligned}
        \overline{g}(x) := \int_{\mathds{R}^{m}}g(x,y)\,\text{d}\mu_{x}(y)\quad \text{and} \quad
        \overline{\ell}(s,x,{ \upnu}) := \int_{\mathds{R}^{m}}\ell(s,x,y,{ \upnu (y)})\,\text{d}\mu_{x}(y) .
    \end{aligned}
\end{equation}
Note that we can now rewrite the effective Hamiltonian $\overline{H}$ as the Bellman Hamiltonian of such effective optimal control problem
\[
 \overline{H}(t,x,p,P) =  \min\limits_{\nu \in U^{ex}}
\;  \left[    -\text{trace}(\bar\sigma \bar\sigma(x,\nu)^{\top}P) - \bar f(x,\nu)\cdot p - \bar \ell(t, x,\nu) \right].
\]

We want to identify the  solution of \eqref{eff-Hamiltonian} with the value function of the effective problem. In order to fit the classical setting of stochastic optimal control, we need to strengthen our previous assumptions as follows.

\noindent
 \textit{Assumption \textbf{(A')}:} It includes all the conditions  in \textbf{(A)}, 
 with $U\subseteq \mathds{R}^k$ for some $k$ and \textbf{(A1)} complemented 
 by the existence of constants $0< \alpha\leq 1 $, $C'_R\geq 0$ such that for all $x \in \mathds{R}^{n},\; y \in \mathds{R}^{m},\; u_{i} \in U$, $i=1,2$,\; $|x|\leq R$, 
\begin{equation}
\label{assumption-dyn 2}
    |f(x, y, u_{1}) - f(x, y, u_{2})| \leq \; C'_R \big(1+|y| \big)\,|u_{1} - u_{2}|^{\alpha} ,
\end{equation}
\begin{equation}\label{assumption-dyn 3}
    |\sigma(x,y,u_{1}) - \sigma(x,y,u_{2})| \leq C_{R}' (1+ |y|) |u_{1} - u_{2}|^{\alpha}.
\end{equation}
This assumption adds to  \textbf{(A)} the H\"older continuity in $u$ of the drift $f$ and of the diffusion $\sigma$, locally in $(x,y)$.

\noindent{\textit{Assumption \textbf{(B')}:}  It includes all the conditions  in \textbf{(B)}, with $U\subseteq \mathds{R}^k$, and complemented  by the existence of constants $\alpha, \beta \in (0, 1] $, $\bar C_R, C_R'$ depending on $R>0$, $\bar C>0$, and a modulus $\omega$, s.t. for all $t_{i}\in [0,T]$, $x_{i}\in \mathds{R}^{n}, y \in \mathds{R}^{m}, u_{i} \in U$, $i=1,2$, $|x_{i}|\leq R$,
\begin{equation}
\label{assumption-cost 2}
\begin{aligned}
	& |\ell(t_{1},x_{1},y,u_{1}) - \ell(t_{2},x_{2},y,u_{2})| \leq  \bar C_R(1+|y| ) \omega(|t_{1} - t_{2}| + |x_{1} - x_{2}|)\\
	&   \qquad \qquad \qquad \qquad \qquad \qquad \qquad \qquad 
	 + C'_R (1+ |y|) |u_{1} - u_{2}|^{\alpha},\\
	& |g(x_{1},y) - g(x_{2},y)| \leq   \bar C(1+ |x_1|\vee |x_2| + |y|)  |x_{1} - x_{2}|^\beta . 
\end{aligned}
\end{equation}
This assumption adds to  \textbf{(B)} a suitable uniform continuity or local  H\"older continuity of $\ell$ and $g$.

The next assumption concerns the invariant probability measure $\mu_{x}$ introduced in Proposition \ref{prop:innv}.

\noindent{\textit{Assumption \textbf{(D)}:}} For each fixed $x$,  the invariant measure satisfies $\mu_{x}(\text{d} y) = m(x,y) \text{d}y$ and its density $m(x,y)$ is such that for all $R>0$, $\exists\, C_{R}>0$ such that
\begin{equation}\label{eq:m zero}
	0 \leq m(x,y) \leq C_{R}\,m_{\circ}(y),\quad \forall \, |x|\leq R
\end{equation}
where $m_{\circ}$ is the density of a positive finite measure. 

The existence of a density $m(x,y)$ for the invariant measure $\mu_{x}(\text{d}y)$ is guaranteed with assumption \textbf{(A)}; see \cite[Corollary 1.4]{bogachev2000generalization}. Then assumption \textbf{(D)} is automatically satisfied in the case of assumption \textbf{(C)(b)}, because $m$ is constant in $x$. 

\begin{rem}\label{rem: D model}
In the motivating model problems of Sections \ref{LEDR} and \ref{sec: app} all these assumptions are satisfied if $\nabla \phi$ is Lipschitz. The density in assumption \textbf{(D)} is $m_{\circ}(y)=\exp(-\beta \phi(y))$. A more general case where \eqref{eq:m zero} holds is when the drift of the process \eqref{fast subsys} is a gradient, i.e., $b(x,y)= \nabla_{y}V(x,y)$, satisfying assumptions \textbf{(A)}, with $V(x,y) = v_{0}(y) + v_{1}(x,y) + v_{2}(x)$. In this case, the invariant measure is $m(x,y)\text{d}y= \exp[V(x,y)]\text{d}y$, see \cite{bogachev2010invariant}.
If $v_{2}(\cdot)$ is continuous, $v_{1}(\cdot)\leq 0$, and $y\mapsto\exp(v_{0}(y)) \in L^{1}(\mathds{R}^{m})$, then setting $m_{\circ}(y) = \exp(v_{0}(y))$ one has
\begin{equation*}
    m(x,y) \leq \left(\max\limits_{|x|\leq R}e^{v_{2}(x)}\right)e^{v_{0}(y)} = C_{R}\, m_{\circ}(y).
\end{equation*}
\end{rem}

Finally, the control set we shall consider is 
\begin{equation}\label{topo control}
	U^{ex}=L^{\infty}(\mathds{R}^{m},U) \quad \text{endowed with the norm } \quad \|\cdot\|_{L^{2\alpha}_{m_{\circ}}} ,
\end{equation}
 where $m_{\circ}$ is defined in assumption \textbf{(D)}, and $\alpha$ is the one in \eqref{assumption-dyn 2}. We recall the weighted Lebesgue space for $r\in [1,\infty)$
\begin{equation*}
	L^{r}_{m_{\circ}} := \left\{ \varphi: \mathds{R}^{m}\to \mathds{R} \text{ s.t. } \, \|\varphi\|^{r}_{L ^{r}_{m_{\circ}}}:= \int_{\mathds{R}^{m}} |\varphi(y)|^{r}\, m_{\circ}(y)\,\text{d}y <+\infty  \right\}.
\end{equation*}
Note that $\big(U^{ex},  \|\cdot\|_{L^{2\alpha}_{m_{\circ}}}\big)$ is a separable complete normed space, and recall  $\;\mathcal{U}^{ex}$ is the set of progressively measurable processes taking values in $U^{ex}$.

\begin{prop}
\label{prop eff dyn} 
Assume (A'), (C), and (D). Then there exist constants $C$ and, for all $R>0$, $K_R$ such that
\begin{equation}
\label{cont_fbar}
	\big|\bar{f}(x_{1},\nu_{1}) - \bar{f}(x_{2},\nu_{2})\big| \leq C\, |x_{1} - x_{2}| 
	 + K_{R}\|\nu_{1}-\nu_{2}\|_{L^{2\alpha}_{m_{\circ}}}^{\alpha}, \quad \text{ for } |x_{1}|, |x_{2}|\leq R ,
\end{equation}
\begin{equation}
\label{cont_sbar}
\big|\bar{\sigma}(x_{1},\nu_{1}) - \bar{\sigma}(x_{2},\nu_{2})\big|  
	\leq C\, |x_{1} - x_{2}| 
	 + K^{1/2}_{R}\|\nu_{1}-\nu_{2}\|_{L^{2\alpha}_{m_{\circ}}}^{\alpha/2}, \quad \text{ for } |x_{1}|, |x_{2}|\leq R ,
\end{equation}
\begin{equation}
\label{linear_g}
    \big|\bar{f}(x,\nu)\big| \leq C(1+ |x|) , \qquad \big|\bar{\sigma}(x,\nu)\big| \leq C(1+ |x|) \quad \forall\, x, \nu .
\end{equation}
Consequently, for each $\nu_.\in \mathcal{U}^{ex}$ the stochastic differential equation \eqref{eff-dynamics} has a unique strong solution. 
\end{prop}

\begin{proof} Here $C$ denotes any positive constant which may change from line to line and from an estimate to another, and is only depending on the data (through the constants in the standing assumptions). We begin with the effective drift. 

Let $(x_{i}, \nu_{i}) \in \mathds{R}^{n}\times L^{\infty}(\mathds{R}^{m},U)$, $i=1,2$. We have
\begin{equation*}
    \begin{aligned}
        \bar{f}(x_{1},\nu_{1}) - \bar{f}(x_{2},\nu_{2}) & = \int_{\mathds{R}^{m}} f(x_{1},y,\nu_{1}(y))\,\text{d} \mu_{x_{1}}(y) - \int_{\mathds{R}^{m}} f(x_{2},y,\nu_{1}(y))\,\text{d} \mu_{x_{1}}(y) \\
        & \quad + \int_{\mathds{R}^{m}} f(x_{2},y,\nu_{1}(y))\,\text{d} \mu_{x_{1}}(y) - \int_{\mathds{R}^{m}} f(x_{2},y,\nu_{2}(y))\,\text{d} \mu_{x_{1}}(y)\\
        & \quad + \int_{\mathds{R}^{m}} f(x_{2},y,\nu_{2}(y))\,\text{d} \mu_{x_{1}}(y) - \int_{\mathds{R}^{m}} f(x_{2},y,\nu_{2}(y))\,\text{d} \mu_{x_{2}}(y)\\
        & = : \text{(I)} + \text{(II)} + \text{(III)}.
    \end{aligned}
\end{equation*}
We have the following
\begin{equation*}
\begin{aligned}
    |\text{I}| & \leq \int_{\mathds{R}^{m}} |f(x_{1},y,\nu_{1}(y)) - f (x_{2},y,\nu_{1}(y))|\, \text{d} \mu_{x_{1}}(y) \leq C\,|x_{1} - x_{2}|.
\end{aligned}
\end{equation*}
Next, for $R>0$ such that $|x_{1}|,|x_{2}| \leq R$, using assumptions (A') and (D), as well as Cauchy-Schwarz inequality, we get	
\begin{equation*}
\begin{aligned}
    |\text{II}| & \leq \int_{\mathds{R}^{m}} |f(x_{2},y,\nu_{1}(y)) - f(x_{2},y,\nu_{2}(y))|\,\text{d} \mu_{x_{1}}(y)\\
    & \leq C'_R\int_{\mathds{R}^{m}} \big(1+  |y|\big)|\nu_{1}(y) - \nu_{2}(y)|^{\alpha}\, \text{d} \mu_{x_{1}}(y)\\
	& \leq C'_R C_R\left( 
	\|\nu_{1}-\nu_{2}\|_{L^{\alpha}_{m_{\circ}}}^{\alpha}  +   \mathfrak{m}_{2}^{1/2}(x_{1})\left(\int_{\mathds{R}^{m}}|\nu_{1}(y) - \nu_{2}(y)|^{2\alpha}\,m_{\circ}(y)\text{d}y 
	\right)^{\frac{1}{2}}\right)\\
	& \leq K_{R} \|\nu_{1} - \nu_{2}\|^{\alpha}_{L^{2\alpha}_{m_{\circ}}}
	\end{aligned}
\end{equation*}
where $K_{R}>0$ is a constant depending on $R$, and where we have used \eqref{eq: moment growth} and the fact that $\|\cdot\|_{L^{\alpha}_{m_{\circ}}} \leq C \,\|\cdot\|_{L^{2\alpha}_{m_{\circ}}}$ with $C>0$ depending only on the total mass of $m_{\circ}$ assumed to be finite with assumption (D).

To estimate (III) it is useful to introduce the densities and write $\text{d}\mu_{x_{1}}(y)=m_1(y)\text{d}y$,  $\text{d}\mu_{x_{2}}(y)=m_2(y)\text{d}y$.
\begin{equation*}
\begin{aligned}
    |\text{III}| & \leq \int_{\mathds{R}^{m}} |f(x_{2},y,\nu_{2}(y)) (m_1-m_2)(y)|\,\text{d} y
     \\& \leq C\, \int_{\mathds{R}^{m}}(1+|x_{2}| + |y|)\,(m_1-m_2)^+(y)\,\text{d} y   
     \\ & \qquad + C\, \int_{\mathds{R}^{m}}(1+|x_{2}| + |y|)\,(m_1-m_2)^-(y)\,\text{d} y   
     =: \text{(III.1)} + \text{(III.2)}.
\end{aligned}
\end{equation*}
Call $\Omega:=\{ y\in {\mathds{R}^{m}} : m_1(y)\geq m_2(y) \}$ and $\phi(y):= C(1+|x_{2}| + |y|)$. Let $\pi$ be any probability measure on $\mathds{R}^{m}\times \mathds{R}^{m}$ with marginals $\mu_{x_{1}}$ and $\mu_{x_{2}}$.
Then 
\begin{equation*}
\begin{aligned}
\text{(III.1)} &= \int_\Omega \phi(y) \,(m_1-m_2)(y)\,\text{d} y   = \iint_{\Omega\times\Omega} (\phi(y) - \phi(y')) \,\text{d}\pi (y, y')
\\ & \leq C  \iint_{{\mathds{R}^{m}}\times{\mathds{R}^{m}}}|y - y'| \,\text{d}\pi (y, y') .
\end{aligned}
\end{equation*}
Since $\pi$ is arbitrary with the given marginals we get
\begin{equation*}
\text{(III.1)} \leq C \mathcal{W}_{1}(\mu_{x_{1}},\mu_{x_{2}})\;\leq C \mathcal{W}_{2}(\mu_{x_{1}},\mu_{x_{2}}) .
\end{equation*}
Using \cite[Corollary 1]{bogachev2014kantorovich}, we have
\begin{equation}\label{eq: est W2}
    \mathcal{W}_{2}(\mu_{x_{1}},\mu_{x_{2}}) \leq C\, \left( \int_{\mathds{R}^{m}} \big| b(x_{1},y) - b(x_{2},y) \big|^{2} \text{d}\, \mu_{x_{2}}(y) \right)^{\frac{1}{2}}
\end{equation}
hence, using Lipschitz continuity of the drift $b$ in $x$, we get
\begin{equation}\label{eq: TV diff 2}
    \text{(III.1)} \leq C\, |x_{1} - x_{2}|.
\end{equation}
The  term  \text{(III.2)} can be estimated in the same way by replacing $\Omega$ with $\{ m_1\leq m_2 \}$. Finally, by exchanging the roles of $(x_{1},\nu_{1})$ and $(x_{2},\nu_{2})$ and putting the estimates together, we obtain \eqref{cont_fbar}.

Next we check  the Lipschitz property of the effective diffusion. In case of assumption (C)(a) we have $\bar{\sigma}(x)=\sigma (x)$ which is Lipschitz by Assumption (A).  
In the case (C)(b), the invariant measure $\mu$ is independent of $x$, so we have
\begin{equation*}
   \bar{\sigma}\bar{\sigma}^{\top}(x,\nu) :=   \bar{\Sigma}(x,\nu) := \int_{\mathds{R}^{m}}\sigma\sigma^{\top}(x,y,\nu(y))\,\text{d}\mu(y)= \int_{\mathds{R}^{m}}\Sigma(x,y,\nu(y))\,\text{d}\mu(y).
\end{equation*}
By the condition (C)(b)(i) we can differentiate the integral  w.r.t. $x$ under the integral sign (by means of dominated convergence theorem). 
Then (C)(b)(ii) ensures that $ \bar{\Sigma}$ has bounded second derivatives in $x$, uniformly in $\nu$. Therefore it has a square root $\bar{\sigma}, $ Lipschitz in $x$ uniformly in $\nu$, by \cite[Theorem 5.2.3, p.132]{stroock1997multidimensional}. 
The linear growth in $x$ uniform in $\nu$ follows easily.
 
Now we turn to the continuity of  $\Bar{\sigma}$ with respect to $\nu$. In the case (C)(a), $\Sigma$ does not depend on $u$ and so $\Bar{\sigma}$ is constant in $\nu$. For the case (C)(b), given $R>0$, we fix $x\in \mathds{R}$ such that $|x|\leq R$. Let $\nu_{i} \in U^{ex}$, $i=1,2$. 
Noting that $\bar{\Sigma}$ is a symmetric positive semidefinite matrix, the following inequality holds (see, e.g. \cite[inequality (3.2)]{wihler2009holder} or \cite[inequality (X.2), p. 290]{bhatia1997matrix}):
\begin{equation}\label{eq:matrix ineq}
    \left| \bar{\sigma}(x,\nu_{1}) - \bar{\sigma}(x,\nu_{2}) \right|  = \left| \bar{\Sigma}^{\frac{1}{2}}(x,\nu_{1}) - \bar{\Sigma}^{\frac{1}{2}}(x,\nu_{2}) \right| \leq C \, \left| \bar{\Sigma}(x,\nu_{1}) - \bar{\Sigma}(x,\nu_{2}) \right|^{\frac{1}{2}}
\end{equation}
where $C>0$ is a constant depending on the dimension $n$ only, and the norm is $|\Sigma|^{2} = \text{trace}\left(\Sigma\Sigma^{\top}\right)$. Therefore we have
\begin{equation*}
\begin{aligned}
    & \left| \bar{\Sigma}(x,\nu_{1}) - \bar{\Sigma}(x,\nu_{2}) \right| 
    \\ & \quad 
    \leq \left|\int_{\mathds{R}^{m}}\sigma\sigma^{\top}(x,y,\nu_{1}(y))\,\text{d}\mu(y) - \int_{\mathds{R}^{m}}\sigma(x,y,\nu_{1}(y))\sigma^{\top}(x,y,\nu_{2}(y))\,\text{d}\mu(y)\right| 
    \\
    & \quad \quad \quad \quad \quad + \left|\int_{\mathds{R}^{m}}\sigma(x,y,\nu_{1}(y))\sigma^{\top}(x,y,\nu_{2}(y))\,\text{d}\mu(y) - \int_{\mathds{R}^{m}}\sigma\sigma^{\top}(x,y,\nu_{2}(y))\,\text{d}\mu(y)\right|
    \\
    & \quad \leq \int_{\mathds{R}^{m}}C_{R}(1+|y|)\big|\sigma(x,y,\nu_{1}(y))-\sigma(x,y,\nu_{2}(y))\big|\,\text{d}\mu(y) \\
    & \quad \quad \quad \quad \quad + \int_{\mathds{R}^{m}} C_{R}(1+|y|) \big|\sigma(x,y,\nu_{1}(y))-\sigma(x,y,\nu_{2}(y))\big|\,\text{d}\mu(y)\\
    & \quad \leq  \int_{\mathds{R}^{m}} 2C_{R}C_{R}'(1+|y|)^{2} \big|\nu_{1}(y)-\nu_{2}(y)\big|^{\alpha}\,\text{d}\mu(y)
\end{aligned}
\end{equation*}
where we have used, respectively, \eqref{assumption-slow} and \eqref{assumption-dyn 3} in the last two inequalities. 
We conclude using Cauchy-Schwarz inequality together with \eqref{eq:m zero} which yield
\begin{equation*}
\begin{aligned}
        \left| \bar{\Sigma}(x,\nu_{1}) - \bar{\Sigma}(x,\nu_{2}) \right| \leq K_{R} \,\|\nu_{1} - \nu_{2}\|^{\alpha}_{L^{2\alpha}_{m_{\circ}}}
\end{aligned}
\end{equation*}
where $K_{R}>0$ is a constant depending on $R$ and on the $4$th moment of $\mu$ (the latter being independent of $x$ thanks to assumption (C)(b)). Back to \eqref{eq:matrix ineq}, we finally have the desired inequality
\begin{equation*}
    \left| \bar{\sigma}(x,\nu_{1}) - \bar{\sigma}(x,\nu_{2}) \right| \leq K_{R}^{1/2} \,\big\|\nu_{1} - \nu_{2}\big\|^{\alpha/2}_{L^{2\alpha}_{m_{\circ}}}.
\end{equation*}
To prove \eqref{linear_g} we use \eqref{assumption-slow} to get
\begin{equation*}
    \begin{aligned}
        & \big|\bar{f}(x,\nu)\big| \leq  \big|\bar{f}(0,\nu)\big| + C |x| \leq \int_{\mathds{R}^m} C(1+|y|) \text{d}\mu_0(y) + C |x| \leq C +\mathfrak{m}_1(0) + C |x|, \\
        & \text{and }\; \big|\bar\sigma\bar\sigma^{\top}(x,\nu)\big| \leq  \big|\bar\sigma\bar\sigma^{\top}(0,\nu)\big| + C^2 |x|^2 
\leq C^2 +\mathfrak{m}_2(0) + C |x|^2.
    \end{aligned}
\end{equation*}

Finally, the well-posedeness of the SDE in \eqref{eff-dynamics} is a direct consequence of the latter properties; see e.g. \cite[Corollary 6.4 in Chap. 1, p. 44]{yong1999stochastic}.
\end{proof}

\begin{prop}
\label{prop eff payoff}
Assume (B') and (D). Then the effective data defined by \eqref{eff payoff data} satisfy the following inequalities, for $|x_1|,|x_2| \leq R$, \vspace*{-0.5em}
\begin{equation}
\label{effl_Lip}
     |\bar{\ell}(t_{1},x_{1},\nu_{1}) - \bar{\ell}(t_{2},x_{2},\nu_{2})| 
 \leq C_R\big( \omega(|t_{1}-t_{2}| + |x_{1}-x_{2}|)  + \|\nu_{1}-\nu_{2}\|_{L^{2\alpha}_{m_{\circ}}}^{\alpha}  \big) , \vspace*{-0.5em}
\end{equation}
\begin{equation}
\label{effl_quad}
    |\bar{\ell}(t,x,\nu)|  \leq  C(1+|x|^{2}) , \vspace*{-0.5em}
\end{equation}
\begin{equation}
\label{effg_Hol}
    |\bar g(x_1) - \bar g(x_2)|\leq C(1+|x_1|\vee |x_2|)|x_1 - x_2|^\beta , \quad  |\bar g(x)|  \leq  C(1+|x|^{2}) .
\end{equation}
\end{prop}

\begin{proof}
Here $C_R$ denotes a  constant which may change from line to line, and only depends on the data of the problem and on $R$. 
For $(t_{i},x_{i},\nu_{i})\in [0,T] \times \mathds{R}^{n}\times L^{\infty}(\mathds{R}^{m},U)$, $i=1,2$, \; $|x_i|\leq R$, we have
\begin{equation*}
\begin{aligned}
    & \bar{\ell}(t_{1},x_{1},\nu_{1}) - \bar{\ell}(t_{2},x_{2},\nu_{2})  \\
    & \quad \quad \quad = \int_{\mathds{R}^{m}}\ell(t_{1},x_{1},y,\nu_{1}(y))\,\text{d} \mu_{x_{1}}(y) 
    \!-\! \int_{\mathds{R}^{m}}\ell(t_{2},x_{2},y,\nu_{1}(y))\,\text{d} \mu_{x_{1}}(y)\\
    &  \quad \quad \quad \quad \quad \quad + \int_{\mathds{R}^{m}}\ell(t_{2},x_{2},y,\nu_{1}(y))\,\text{d} \mu_{x_{1}}(y) \!-\! \int_{\mathds{R}^{m}}\ell(t_{2},x_{2},y,\nu_{2}(y))\,\text{d} \mu_{x_{1}}(y)\\
    &  \quad \quad \quad \quad \quad \quad + \int_{\mathds{R}^{m}}\ell(t_{2},x_{2},y,\nu_{2}(y))\,\text{d} \mu_{x_{1}}(y) \!-\! \int_{\mathds{R}^{m}}\ell(t_{2},x_{2},y,\nu_{2}(y))\,\text{d} \mu_{x_{2}}(y)\\
    & \quad \quad \quad =  \text{(I)} + \text{(II)}+ \text{(III)} .
\end{aligned}
\end{equation*}
By \eqref{assumption-cost 2} and \eqref{eq: moment growth}
\begin{equation*}
    \text{(I)}  \leq \bar C_R\big(1+\mathfrak{m}_{1}(x_{1})\big)\,\omega(|t_{1}-t_{2}|+|x_{1}-x_{2}|) \leq C_R \omega(|t_{1}-t_{2}|+|x_{1}-x_{2}|) .
\end{equation*}
Next, we use assumption (D) as in the estimate of  the term (II) in the proof of Proposition \ref{prop eff dyn} to get 
\begin{equation*}
    \text{(II)} 
	 \leq C'_R \int_{\mathds{R}^{m}} (1+ |y|)|\nu_{1}(y) - \nu_{2}(y)|^{\alpha}\, \text{d}\mu_{x_{1}}(y) \leq K_{R} \|\nu_{1} - \nu_{2}\|_{L^{2\alpha}_{m_{\circ}}}^{\alpha}
\end{equation*} 
For  the last term we use \eqref{assumption-cost} and proceed as for the term (III) in the proof of Proposition \ref{prop eff dyn} by means of the densities $m_1, m_2$:
\begin{equation*}
\begin{aligned}
  |\text{III}| & \leq \int_{\mathds{R}^{m}} \big| \ell(t_2, x_{2},y,\nu_{2}(y)) (m_1-m_2)(y)\big|\,\text{d} y
     \\& \leq K\, \int_{\mathds{R}^{m}}(1+|x_{2}|^2 + |y|)\,(m_1-m_2)^+(y)\,\text{d} y   
     \\ &+ K\, \int_{\mathds{R}^{m}}(1+|x_{2}|^2 + |y|)\,(m_1-m_2)^-(y)\,\text{d} y   
     =: \text{(III.1)} + \text{(III.2)}.
\end{aligned}
\end{equation*}
The integrals in (III.1) and (III.2) are estimated as the corresponding terms  in the proof of Proposition \ref{prop eff dyn} by $C |x_{1}-x_{2}|$. Summing up the upperbounds of the three terms (I), (II), and (III),  and exchanging the roles of $(t_{i},x_{i},\nu_{i})$, $i=1,2$, we get
\begin{equation*}
\begin{aligned}
    & |\bar{\ell}(t_{1},x_{1},\nu_{1}) - \bar{\ell}(t_{2},x_{2},\nu_{2})| \\
    & \quad \leq C_R\big( \omega(|t_{1}-t_{2}| + |x_{1}-x_{2}|)  + |x_{1}-x_{2}| + \|\nu_{1}-\nu_{2}\|_{L^{2\alpha}_{m_{\circ}}}^{\alpha} \big) .
\end{aligned}
\end{equation*}
To check the growth \eqref{effl_quad}, let $(t,x,\nu)\in [0,T] \times\mathds{R}^{m}\times U^{ex}$ and use \eqref{assumption-cost} to get 
\begin{equation*}
\begin{aligned}
    |\bar{\ell}(t,x,\nu)| & \leq \int_{\mathds{R}^{m}} |\ell(t,x,y,\nu(y))|\,\text{d} \mu_{x}(y)  \leq K\left(1 + |x|^{2} + \int_{\mathds{R}^{m}} |y|\,\text{d} \mu_{x}(y)\right)\\
    & = K(1 + |x|^{2} + \mathfrak{m}_{1}(x)) \leq C(1+|x|^{2})
\end{aligned}
\end{equation*}
where in the last inequality we have used \eqref{eq: moment growth}. 

The proof of \eqref{effg_Hol} is analogous, by using the second inequality in \eqref{assumption-cost 2}.
\end{proof}
}}

We can now define the \underline{effective optimal control problem}
\begin{equation}
    \label{eff-ocp}
    \tag{$\;\overline{OCP}\;$}
    V(t,x) := \sup_{\upnu_{\cdot} \in \mathcal{U}^{ex}} \overline{J}(t,x,\upnu_{\cdot}
    ),\quad \text{subject to }\; \eqref{eff-dynamics}
\end{equation}
where $\;\mathcal{U}^{ex}$ is the set of progressively measurable processes taking values in the extended control set $U^{ex}$ defined in \eqref{topo control}, the effective payoff is
\begin{equation}
    \label{eff-pay off}
    \begin{aligned}
        \overline{J}(t,x,\upnu_{\cdot}(\cdot)) & :=  \mathds{E}\left[ 
    e^{\lambda(t-T)} \overline{g}(\hat{X}_{T}) + 
   \int_{t}^{T} \overline{\ell}(s,\hat{X}_{s},\upnu_{s}) e^{\lambda(t-s)} \,\text{d}s \; \bigg|\; \hat{X}_{t}=x
    \right]\,,
    \end{aligned}
\end{equation}
and the process $\hat{X}_{s}$ solves \eqref{eff-dynamics}.

\begin{thm}\label{thm: value function sol limit pde}
Assume (A'), (B'), (C), and (D). 
Then the value function $V(t,x)$ of \eqref{eff-ocp} is the unique continuous viscosity solution to the Cauchy problem \eqref{CP-limit HJB} satisfying the growth condition \eqref{qg}. In particular, it is the limit of the value functions $V^{\varepsilon}$ defined in \eqref{value function}.
\end{thm}

We need the following lemma to prove the latter theorem. 

\begin{lem}
Under the assumptions of Theorem \ref{thm: value function sol limit pde}  the value function $V(t,x)$ satisfies, for some constant $C$,\vspace*{-0.5em}
\begin{equation}
\label{V_quad}
| V(t,x)| \leq C(1+|x|^2)\quad \forall\, (t,x)\in[0,T]\times\mathds{R}^n , \vspace*{-0.5em}
\end{equation}
and \vspace*{-0.5em}
\begin{equation}
\label{limT}
\lim_{(t,x)\to(T, \bar x)} V(t,x) = \bar g (\bar x) .
\end{equation}
\end{lem}
\begin{proof} 
The properties of $\bar f$ and $\bar\sigma$ in Proposition \ref{prop eff dyn} allow to use standard estimates on the moments of the process defined by \eqref{eff-dynamics}, see e.g.  \cite[Appendix D]{fleming2006controlled}, and
imply the following inequalities, for suitable constants $C$ depending only on the data, 
\begin{equation}
\label{moment1}
\mathds{E} \left[ \sup_{s\in [t,T]} |\hat{X}_{s}-x|^2 \right] \leq C e^{C(T-t)} \int_t^T(C+|x|^2) ds \leq C_T(T-t)(1+|x|^2) ,\vspace*{-1em}
\end{equation} 
\begin{equation}
\label{moment2}
\mathds{E} \left[ \sup_{s\in [t,T]} |\hat{X}_{s}|^2 \right] \leq C(1+|x|^2) .
\end{equation}
Then \eqref{effl_quad} and the second inequality in \eqref{effg_Hol}, together with \eqref{moment2} give \eqref{V_quad}. To prove \eqref{limT} we assume for simplicity $\lambda =0$ and compute
\begin{equation*}
\begin{aligned}
| \overline{J}(t,x,\upnu_{\cdot}) - \bar g(\bar x)| & \leq \mathds{E} \left[ |\bar g(\bar x)- \bar g(x)| + |\bar g(\hat{X}_{T})- \bar g(x)| +  \int_t^T |\bar \ell(s, \hat{X}_{s}, \upnu_{s})| ds \right]
\\ 
& =: \text{(I)} + \text{(II)} + \text{(III)} \,.
\end{aligned}
\end{equation*} 
By \eqref{effg_Hol}, we have $\; |\bar g(\bar x) - \bar g(x)|\leq C(1+|\bar x|\vee |x|)|\bar x - x|^\beta,\; $ so $\text{(I)} \to 0$ as $x\to\bar x$. Moreover, by H\"older inequality with $p=2/\beta$,
\begin{equation*}
\begin{aligned}
\text{(II)} &\leq \mathds{E} \left[ C(1+ |\hat{X}_{T}|\vee |x|) |\hat{X}_{T}  - x|^\beta \right] \leq C \mathds{E} \left[ (1+|\hat{X}_{T}|\vee |x|)^{p'}\right]^{\frac{1}{p'}}  \mathds{E} \left[ |\hat{X}_{T}  - x|^2 \right]^{\frac{\beta}{2}}
 \\
& \leq C \left(1+ |x|^{p'} + \mathds{E} \left[ |\hat{X}_T |^2 \right] \right)  \mathds{E} \left[ |\hat{X}_T  - x|^2 \right]^{\frac{\beta}{2}} \leq C(1+|x|^2)  \mathds{E} \left[ |\hat{X}_T  - x|^2 \right]^{\frac{\beta}{2}} ,
\end{aligned}
\end{equation*}
where we used that $p'\leq 2$ and \eqref{moment2}. Then $\text{(II)} \to 0$ as $t\to T_{-}$ by \eqref{moment1}.

Finally we use  \eqref{effl_quad} and \eqref{moment2} to estimate
\begin{equation*}
\text{(III)} \leq (T-t) C \left(1+ \mathds{E} \left[ \sup_{s\in [t,T]} |\hat{X}_{s}|^2 \right] \right) \leq (T-t) C (1+|x|^2) ,
\end{equation*}
and so also $\text{(III)}\to 0$ as $t\to T_{-}$, which completes the proof of \eqref{limT} by the arbitrariness of $\upnu_{\cdot} \in \mathcal{U}^{ex}$.
\end{proof}

\begin{proof} [Proof (of Theorem \ref{thm: value function sol limit pde})] 
The proof is based on the Dynamic Programming Principle, namely
\begin{equation}
\label{dpp}
 V(t,x) := \sup_{\upnu_{\cdot} \in \mathcal{U}^{ex}} \mathds{E}\left[ 
    e^{\lambda(t-\theta)
    } V(\theta, \hat{X}_{\theta}) + 
   \int_{t}^{\theta} \overline{\ell}(s,\hat{X}_{s},\upnu_{s})e^{\lambda(t-s)}\,\text{d}s \; \bigg|\; \hat{X}_{t}=x
    \right]\,,
\end{equation}
for all stopping times $\theta$ valued in $[t, T]$. This is usually proved for the weak formulation of the control problem and then extended to the strong formulation by proving that the value functions coincide. This is well-known in the compact case; for unbounded problems including ours we refer to \cite{karoui2013capacities} and to the very general treatment in \cite{djete2022mckean}.

Next, one deduces from the two inequalities in \eqref{dpp} that the lower semicontinuous envelope $V_*$ is a supersolution of the HJB equation in \eqref{CP-limit HJB}  and the upper semicontinuous envelope $V^*$ is a subsolution, see, e.g., \cite{da2006uniqueness} or \cite{yong1999stochastic}. Moreover they satisfy $V_*(T,x)=V^*(T,x)$ for all $x$ by \eqref{limT}, and have at most quadratic growth by \eqref{V_quad}. Then we can use the comparison Theorem 2.1 in \cite{da2006uniqueness} to get $V^*\leq V_*$, which implies that $V$ is continuous and the unique solution of the Cauchy problem \eqref{CP-limit HJB} satisfying \eqref{qg}. The convergence to $V$ of the value functions $V^{\varepsilon}$ now follows from Theorem \ref{thm-conv-value}.
\end{proof}


\begin{rem}
\label{rem:W}
The system \eqref{eff-dynamics} can also be restricted to standard control functions  $u_\cdot \in  \mathcal{U}$, because they are extended controls constant in $y$. Consider the corresponding value function $W(t,x) := \sup_{u_{\cdot} \in \mathcal{U}} \overline{J}(t,x,u_{\cdot})$. Then
\[
\lim_{\varepsilon\to 0} V^\varepsilon(t,x,y)\geq W(t,x) ,
\]
because the left hand side is $V(t,x)$, which is larger than $W(t,x)$, being a sup over a larger set $\mathcal{U}^{ex}\supseteq \mathcal{U}$. This means that, when using standard controls, the perturbed system \eqref{dynamics} can give a better performance than the limit one \eqref{eff-ocp}. Theorem \ref{thm: practice}, that we prove in Section \ref{sec:DR}, is based on this remark. 

In some cases one can have $V=W$ and the extended controls are not necessary: this occurs when the controls are decoupled from the fast variables $Y$ in the data $f, \sigma$, and $\ell$, as in  Remark \ref{rem: decoup1}.
\end{rem}

\section{Convergence of  the trajectories}
\label{sec: conv traj 2}

We have shown so far that the value function $V^{\varepsilon}$ in \eqref{value function} converges locally uniformly to the value function $V$ in \eqref{eff-ocp} as $\varepsilon\to 0$. In this section, we are interested in the link between the singularly perturbed dynamics \eqref{dynamics} and the corresponding effective one \eqref{eff-dynamics}. 
Mainly we will show that, under the standing assumptions and if $\sigma = 0$ in (A2), 
as $\varepsilon\to 0$ every solution to \eqref{eff-dynamics} is approximated by a sequence of processes of the form \eqref{dynamics}, in a sense that we make precise, and, conversely, the limit of any converging sequence of trajectories of \eqref{dynamics} solves a relaxation of \eqref{eff-dynamics}.

\subsection{Convergence of trajectories with vanishing diffusion} 
\label{sec:conv traj}

In this subsection, we will assume, besides the standing assumptions of \S \ref{sec: setting}, that $\lambda = 0$ and the limit in (A2)  is null, that is,
\begin{equation}
\label{eq: limit sigma 0}
    \lim\limits_{\varepsilon\to 0}\sigma^{\varepsilon}(x,y,u)=0\quad \text{locally uniformly}.
\end{equation}
In this case, assumption (C) is satisfied and the effective dynamics \eqref{eff-dynamics} becomes the deterministic control system
\begin{equation}
    \label{eff-dynamics 2}
    \left\{
    \begin{aligned}
        \frac{\text{d}\hat{x}_{t}}{\text{d}t} &= \int_{\mathds{R}^{m}} f(\hat{x}_{t},y,\upnu_{t}(y))\text{d}\mu_{\hat{x}_{t}}(y)\\
        {\upnu_{\cdot}}&\in \mathcal{U}^{ex},\quad \text{and }\; \hat{x}_{0} =x\in\mathds{R}^{n}.
    \end{aligned}
    \right.
\end{equation}
Note that the right hand side of the ODE is $\overline f(\hat{x}_{t}, \upnu_{t})$, and, since there is no randomness, $\mathcal{U}^{ex}$ is the set of Lebesgue measurable functions $[0,T]\to U^{ex}$, where $U^{ex}$ is as defined in \eqref{topo control}. 
Now the effective control problem \eqref{eff-ocp} is deterministic and its value function simplifies to
\begin{equation}
    \label{eff-ocpd}
    \tag{$\;\overline{OCP}d\;$}
    V(t,x) := \sup_{\upnu_{\cdot} \in \mathcal{U}^{ex}} \left\{ \overline{g}(\hat{x}_{T}) + 
   \int_{t}^{T} \overline{\ell}(s,\hat{x}_{s},\upnu_{s})\,\text{d}s\right\} ,\quad \text{subject to }\; \eqref{eff-dynamics 2}
\end{equation}

If we define 
\begin{equation}
    \label{barF}
\overline{F}(x):=\overline f(x,U^{ex})
\end{equation}
the effective dynamics \eqref{eff-dynamics 2}  can be equivalently expressed by
\begin{equation}
    \label{eff-dynamics 3}
 \hat{x}_{t_{2}}-\hat{x}_{t_{1}} \in \int_{t_{1}}^{t_{2}} \overline{F}(\hat{x}_{s})\,\text{d}s .
\end{equation}

The next two theorems connect the trajectories of this system to the process 
\begin{equation}
\label{dynamics2}
\left\{\;
\begin{aligned}
    \text{d}X_{t} &= f(X_{t},Y_{t},u_{t})\,\text{d}t + \sqrt{2}\,\sigma^{\varepsilon}(X_{t},Y_{t},u_{t})\,\text{d}W_{t},\quad X_{0} = x \in\mathds{R}^{n} ,\\
    \text{d}Y_{t} & = \frac{1}{\varepsilon}\,b(X_{t},Y_{t})\,\text{d}t + \sqrt{\frac{2}{\varepsilon}}\,\Bar\varrho    \,\text{d}W_{t},\quad Y_{0}=y\in\mathds{R}^{m} .
\end{aligned}
\right.
\end{equation}

\begin{thm}
\label{conv-thm-1}
Assume (A'), (D) and \eqref{eq: limit sigma 0}. Then,  for all $y\in \mathds{R}^m$, any solution $\hat{x}_{\cdot}$ to the controlled effective dynamics \eqref{eff-dynamics 2} has the following property: for all  $\varepsilon>0$ there is  a control $u_\cdot^\varepsilon\in \mathcal{U}$ such that the $x$-component of the corresponding trajectory $(X^{\varepsilon} , Y^{\varepsilon})$ of \eqref{dynamics2}  converges to 
$\hat{x}_{\cdot}$ 
in the sense 
\begin{equation*}
    \lim\limits_{\varepsilon\to 0}  \int_{0}^{T} \mathds{E}\left[\big| X^{\varepsilon}_{s} - \hat{x}_{s}\big|^{2}\right]\,\text{d}s + \mathds{E}\left[|{X}^{\varepsilon}_{T} - \hat{x}_{T}|^{2}\right] = 0.
\end{equation*}
\end{thm}

\begin{proof} 
We exploit the convergence  of the value function  \eqref{value function} with a special choice of the running  payoff $\ell$ and final utility function $g$. 
We fix an initial condition $y\in\mathds{R}^{m}$ for the fast process and  consider a  pair $(\hat{x}_{\cdot},\hat{\upnu}_{\cdot}):[0,T]\to \mathds{R}^{n}\times {U}^{ex}$  satisfying \eqref{eff-dynamics 3} with $\hat{x}_{0}=x\in\mathds{R}^n$ fixed. We  choose a payoff functional of the form \eqref{cost function} with 
\begin{equation}
\label{eq: running cost - proof}
	g(z) = -|z- \hat{x}_{T}|^{2}, \quad \text{ and } \quad \ell(s,z,u) = -\big|z - \hat{x}_{s}\big|^{2}, 
\end{equation}
so $g$ and $\ell$ satisfy assumption (B'). We assume for simplicity the discount factor $\lambda=0$. Then the value function of the optimal control problem \eqref{value function} satisfies 
\begin{equation*}
    \begin{aligned}
        &V^{\varepsilon}(0,x,y)  :=
        \sup\limits_{u_{\cdot}\in\mathcal{U}} \;  \mathds{E}\left[{-|X^{\varepsilon}_{T} - \hat{x}_{T}|^{2}} - \int_{{0}}^{T}\big|X^{\varepsilon}_{s} - \hat{x}_{s}\big|^{2} \, \text{d}s\right]  
    \end{aligned}
\end{equation*}
subject to \eqref{dynamics2}.
The minus sign in the running payoff is due to the fact that we have a maximization problem.
Thanks to the convergence result of the value function in Theorem \ref{thm-conv-value} and to the representation result of Theorem \ref{thm: value function sol limit pde}, we deduce that $V^{\varepsilon}({0},x,y)$ converges locally uniformly to $V({0},x)$ which solves an effective optimal control problem of the form \eqref{eff-ocpd}. Note that here the effective payoff functional remains the same, because $g$ and $\ell$ are independent of the variable $y$.  Since $V({0},x)\leq 0$ and $\hat{x}_{\cdot}$ is an admissible trajectory satisfying  \eqref{eff-dynamics 3} for which the payoff functional is null, we get  $V({0},x) = 0$ and $\hat{x}_{\cdot}$ is optimal. This implies $V^{\varepsilon}({0},x,y)$ converges to $0$ as $\varepsilon\to 0$, i.e., 
\begin{equation*}
    \forall\;\delta>0,\; \exists\; E>0\; \text{ s.t.: }\; \forall\; \varepsilon\leq E,\quad |V^{\varepsilon}({0}, x,y)|\leq \frac{\delta}{2}. 
\end{equation*}
Fix $\delta>0$. For each $\varepsilon >0$ choose a control whose corresponding trajectory 
 $
 (\bar X^{\varepsilon},\bar Y^{\varepsilon})$ 
  is $\frac{\delta}{2}$-suboptimal for $V^{\varepsilon}({0},x,y)$, so that 
\begin{equation*}
\begin{aligned}
	&-\frac{\delta}{2} \leq V^{\varepsilon}({0},x,y) \leq \mathds{E}\left[ -|\bar{X}^{\varepsilon}_{T} - \hat{x}_{T}|^{2} - \int_{{0}}^{T}\big|\bar X^{\varepsilon}_{s} - \hat{x}_{s}\big|^{2} \,\text{d}s\right]+ \frac{\delta}{2}.
\end{aligned}
\end{equation*}
Therefore $ \forall\, \delta>0,\, \exists\, E>0,\, \forall\, \varepsilon\leq E,\,\exists \bar X^{\varepsilon}$ such that
\begin{equation*}
   \, 0 \leq \mathds{E}\left[|\bar{X}^{\varepsilon}_{T} - \hat{x}_{T}|^{2} + \int_{{0}}^{T}\big|\bar X^{\varepsilon}_{s} - \hat{x}_{s}\big|^{2} \,\text{d}s\right]\leq \delta 
\end{equation*}
which is the desired result.
\end{proof}

The next result shows that every deterministic limit in the expected distance of a sequence of  singularly perturbed controlled trajectories 
 is a solution of the convexification of the effective differential inclusion  \eqref{eff-dynamics 3}. 

\begin{thm}\label{conv-thm 2}
Assume (A'), (D) and \eqref{eq: limit sigma 0}. If for a  sequence of controlled processes $(X^{\varepsilon_n}_{\cdot}, Y^{\varepsilon_n}_{\cdot})$ of \eqref{dynamics2} with $\varepsilon_n\to 0$ there is a deterministic process $\overline{x}_{\cdot}$ such that 
\begin{equation}
\label{sense of convergence}
    \lim\limits_{\varepsilon_{n}\to 0} {\int_{{0}}^{T}}\mathds{E}\left[\,\big|X^{\varepsilon_{n}}_{s} - \overline{x}_{s}\big|^{p}\,\right]{\text{d}s} = 0,
\end{equation}
for some $p\in [1,2]$, then $\overline{x}_{\cdot}$ satisfies
\begin{equation*}
	\dot{\overline{x}}_{s} \in \overline{\text{co}}\, \overline{F}(\overline{x}_{s}),\quad \text{a.e. } s\in [0,T],
\end{equation*} 
where $\overline{F}$ is defined by \eqref{F,G} and \eqref{barF}, and $\overline{\text{co}}$ denotes the closed convex hull. 
\end{thm}

\begin{proof} 
Fix  a sequence $X^{\varepsilon}_{\cdot}$ solution to \eqref{dynamics} converging  to $\overline{x}_{\cdot}$ in the sense  of \eqref{sense of convergence}.  We are going  to show that the limit process $\overline{x}_{\cdot}$ can be approximated by a sequence of trajectories  solving \eqref{eff-dynamics 3}. We consider an optimal control problem of the form \eqref{value function} where the running payoff and final utility function  are 
\[
\ell(s,z,u) = -|z-\overline{x}_{s}|^p\,,\quad  \text{ and } \quad
g\equiv 0\,.
\]
 Since $X^{\varepsilon}_{\cdot}$ is an admissible solution to \eqref{value function}, we have\vspace*{-0.5em}
\begin{equation*}
    \int_{{0}}^{T}-\mathds{E}\left[\,\big|X^{\varepsilon}_{s} - \overline{x}_{s}\big|^{p}\,\right]\,\text{d}s \leq V^{\varepsilon}({0},x,y)\leq 0.
\end{equation*}
We deduce from \eqref{sense of convergence} that $V^{\varepsilon}({0},x,y)$ converges to $0$ as $\varepsilon\to 0$. This means that the limit value function $V({0},x)$ of the effective optimal control problem \eqref{eff-ocpd} also equals $0$. 
Hence, one can consider a minimizing sequence $\{x^{k}_{\cdot}\}_{k}$ of the effective problem \eqref{eff-ocp} with $g$ and $\ell$ as above such that\vspace*{-0.5em}
\begin{equation*}
    \int_{{0}}^{T}\,\big|x^{k}_{s} - \overline{x}_{s}\big|^{p}\,\,\text{d}s \xrightarrow[k\to+\infty]{} 0
\end{equation*}
which yields (up to extracting a subsequence) $\; \lim\limits_{k\to +\infty} \big|x^{k}_{s} - \overline{x}_{s}\big|^{p} = 0,\quad \text{a.e. }\, s\in [{0},T]. \;$
Since $x^{k}_{\cdot}$ solves \eqref{eff-dynamics 3}, by a compactness theorem for differential inclusions, e.g.,  \cite[Theorem 4.1.11, p.186]{clarke2008nonsmooth}, we get a subsequence (again denoted by $x^{k}_{\cdot}$) that  converges uniformly to $z_{\cdot}$ and whose derivatives converge weakly to $\dot{z}_{\cdot}$ where
\begin{equation}\label{co F proof}
\dot{z}_{s} \in \overline{\text{co}}\,\overline{F}(z_{s}),\quad \text{a.e. } s\in [{0},T] ,
\end{equation}
which is the convexification of  \eqref{eff-dynamics 3}. The latter theorem holds true because, for every $x$, $\overline{\text{co}}\,\overline{F}(x)$ is a nonempty compact convex set, moreover $\overline{F}(x)$ is upper semicontinuous\footnote{$F$ is upper semicontinuous in $x$ if $\forall\, \varepsilon>0,\, \exists\,\delta>0$ s.t. $|x-x'|\leq \delta \Rightarrow F(x')\subset F(x) + \varepsilon\,B$ where $B$ is the unit ball (see 
\cite[page 39]{aubin2009set}).} as a direct consequence of \cite[Proposition 1.4.14, p.47]{aubin2009set}, hence also $\overline{\text{co}}\,\overline{F}(x)$ is u.s.c., and finally every element of $\overline{\text{co}}\,\overline{F}(x)$ is upper bounded by an affine function of $|x|$ by \eqref{assumption-slow}. Therefore, and when $p\geq 1$, one has
\begin{equation*}
\begin{aligned}
	|z_{s} - \overline{x}_{s}|^{p} & \leq \|x^{k}_{\cdot} - z_{\cdot}\|^{p}_{\infty} + |x^{k}_{s} - \overline{x}_{s}|^{p} \xrightarrow[k\to+\infty]{} 0
\end{aligned}
\end{equation*}
for almost every $s\in[{0},T]$.  Then $z_{s} = \overline{x}_{s}$  and $\overline{x}$ satisfies \eqref{co F proof}.
\end{proof}

The last result in this subsection concerns the quasi-optimality of the approximating sequence in Theorem \ref{conv-thm-1} under the additional conditions that $g$ and $\ell$ are bounded and depend  only on $x$. 
By quasi-optimal we mean a sequence of trajectories $(X^{\varepsilon_{n}}, Y^{\varepsilon_{n}})$ which satisfy \vspace*{-0.5em}
\begin{equation}\label{eq: quasi-opt}
	V^{\varepsilon_{n}}(t,x,y) \leq \mathds{E}\left[g(X^{\varepsilon_{n}}_{T}) + \int_{t}^{T}\ell(X^{\varepsilon_{n}}_{s})\text{d}s\right] + o(1),\quad \text{as } \varepsilon_{n}\to 0.
\end{equation}

\begin{thm}
\label{thm-quasi-opt}
Assume (A), (D), \eqref{eq: limit sigma 0}, and that the utility function $g=g(x)$ and the payoff $\ell=\ell(x)$ are continuous bounded functions of $x$ only. Then, for a solution $(\hat{x},\hat{\upnu})$ of  \eqref{eff-dynamics 2} that is optimal for the effective problem \eqref{eff-ocpd}, the approximating sequence of trajectories $(X^{\varepsilon},Y^{\varepsilon})$ found in Theorem \ref{conv-thm-1} has in addition the quasi-optimality property  \eqref{eq: quasi-opt}. 
\end{thm}

\begin{proof}
Let $(\hat{x},\upnu_{\cdot})$ be an optimal solution of \eqref{eff-ocpd}. 
Let us introduce the path space $C_{n}=C([t,T],\mathds{R}^{n})$ of continuous functions from $[t,T]$ to $\mathds{R}^{n}$ endowed with the uniform topology where the distance between two functions $f$ and $g$ of $s\in [t,T]$ is\vspace*{-0.5em}
\begin{equation*}
    \rho(f,g) = \sup\limits_{s\in[t,T]}\max\limits_{1\leq i\leq n}|f_{i}(s)-g_{i}(s)|.
\end{equation*}
A stochastic process $X_{\cdot}$ is then a random function, that is a mapping from $(\Omega, \mathcal{F}, \mathds{P})$ into $C_{n}$, and such that for each $\omega\in \Omega$, $X_{\cdot}(\omega)$ is an element of $C_{n}$, i.e. a continuous function on $[t,T]$, whose value at $s$ is denoted by $X_{s}(\omega)$. And for fixed $s$, let $X_{s}$ denote the real function on $\Omega$ with value $X_{s}(\omega)$ at $\omega$. Then $\{\mathds{X}^{\varepsilon}\}_{\varepsilon>0}:=\{(X^{\varepsilon}_{s},\; t\leq s \leq T)\}_{\varepsilon>0}$ is a sequence of random functions. We also denote by $\mathbf{\hat{x}}$ the path $\mathbf{\hat{x}}:=(\hat{x}_{s},\,t\leq s\leq T)$. 
From Theorem \ref{conv-thm-1} (with the initial time 0 replaced by $t\in[0,T]$) we have $
(X^{\varepsilon},Y^{\varepsilon})$ 
solution of \eqref{dynamics} with $X^{\varepsilon}_t=x,Y^{\varepsilon}_t=y$, such that \vspace*{-0.5em}
\begin{equation}\label{eq: conv integral E}
    \lim\limits_{\varepsilon\to 0} \int_{t}^{T}\mathds{E}\left[\big| X^{\varepsilon}_{s} - \hat{x}_{s}\big|^{2} \right]\,\text{d}s = 0
\end{equation}
which is the $L^2$ convergence in the path space $C_n$ for the random function $\mathds{X}^{\varepsilon}$ to $\mathbf{\hat{x}}$. In other words, \eqref{eq: conv integral E} can be written as \vspace*{-0.5em}
\begin{equation}\label{eq: conv E tilde}
	\lim\limits_{\varepsilon\to 0} \widetilde{\mathds{E}}[ \,|\mathds{X}^{\varepsilon} - \mathbf{\hat{x}}|^{2}\,] = 0
\end{equation}
where, for a given random function $\mathds{Z}:=(Z_{s}, t \leq s\leq T)$, we used the notation 
\begin{equation*}
	\widetilde{\mathds{E}}[\mathds{Z}] := \frac{1}{T-t}\int_{t}^{T}\mathds{E}[Z_{s}]\text{d}s =  \frac{1}{T-t}\int_{t}^{T}\int_{\Omega}Z_{s}(\omega)\,\text{d}\mathds{P}(\omega)\text{d}s = \int Z_{s}(\omega)\text{d}\mathds{P}\otimes \lambda(s,\omega)
\end{equation*}
the last integral being computed over $\Omega\times[t,T]$ and $\lambda$ is the uniform probability measure on $[t,T]$. 
A direct application of Markov's inequality \cite[(21.12), p.276]{billingsley2008probability} to \eqref{eq: conv E tilde} ensures that $\mathds{X}^{\varepsilon}$ converges to $\mathbf{\hat{x}}$ in probability, which in turn implies the convergence in distribution \cite[Thm. 25.2, p. 330]{billingsley2008probability}, that gives
\begin{equation}\label{eq: conv dist C}
	\widetilde{\mathds{E}}[h(\mathds{X}^{\varepsilon})] \xrightarrow[\varepsilon\to 0]{} \widetilde{\mathds{E}}[h(\mathbf{\hat{x}})] =\frac{1}{T-t}\int_{t}^{T}h(\mathbf{\hat{x}})(s)\text{d}s
\end{equation}
for every $h$ continuous and bounded function of $C_{n}$ into $C_{1}$.  Now define $h(\mathds{X})(s) = \ell(X_s)$. 
Then \eqref{eq: conv dist C} implies 
\begin{equation}
\label{proof: ell}
	\int_{t}^{T}\mathds{E}[\ell(X^{\varepsilon}_{s})]\text{d}s \xrightarrow[\varepsilon\to 0]{} \int_{t}^{T}\mathds{E}[\ell(\hat{x}_{s})]\text{d}s = \int_{t}^{T} \ell(\hat{x}_{s})\text{d}s.
\end{equation}

For the convergence of the final utility function, recall from Theorem \ref{conv-thm-1} that $\; \lim\limits_{\varepsilon\to 0} \mathds{E}\left[|{X}^{\varepsilon}_{T} - \hat{x}_{T}|^{2}\right] = 0.\; $
This implies the convergence in probability via Markov's inequality and hence convergence in distribution.
Since $g$ is bounded and continuous, we get that $\; \lim\limits_{\varepsilon\to 0} \mathds{E}\left[g(X^{\varepsilon}_{T})\right] = g(\hat{x}_{T}). \; $
Combining this with \eqref{proof: ell} and the optimality of $\hat{x}$ we obtain \vspace*{-0.5em}
\begin{equation*}
	 \lim\limits_{\varepsilon\to 0} \, \mathds{E}\left[g(X^{\varepsilon}_{T}) + \int_{t}^{T}\ell(X^{\varepsilon}_{s})\text{d}s\right] = V(t,x) =g(\hat{x}_{T}) + \int_{t}^{T}\ell(\hat{x}_{s})\text{d}s. 
\end{equation*}
Finally we recall from Theorem \ref{thm-conv-value} that $V^{\varepsilon}(t,x,y) \to V(t,x)$ locally uniformly in  $x, y$, as $\varepsilon\to 0$. Then by a triangular inequality we easily get \eqref{eq: quasi-opt}.
\end{proof}

\begin{rem}
The proof of this theorem works as well if $(\hat{x},\hat{\upnu})$ is merely quasi-optimal, i.e., it misses $V(t,x)$ by a quantity $\eta>0$, and gives an approximating sequence that is quasi-optimal for the perturbed problem, in the sense that \eqref{eq: quasi-opt} holds with the addition of  $\eta$ on the right hand side.
\end{rem}

\subsection{The model problem: deep relaxation with learning rate}\label{sec:DR}
Here we apply the results  of Sections \ref{sec:limit V} and \ref{sec:conv traj} to the motivating examples of Sections  \ref{LEDR} and \ref{sec: app}. 
Given $\phi : \mathds{R}^n\to \mathds{R}$, consider its  relaxed gradient descent \eqref{sp-learn}, where the control $u_t\in U\subseteq \mathds{R}$ is the learning rate of the SGD algorithm. Then the effective control problem arising in the singular perturbation limit is \eqref{effective_sp-learn}, i.e., \vspace*{-0.5em}
 \begin{equation}\label{eq proof: limit ocp}
    \begin{aligned}
        \overline{\mathcal{V}}(x) :=  { \min\limits_{\upnu_{\cdot} \in \mathcal{U}^{\text{ex}}}\; \phi(\overline{X}_{T})}, \quad & \text{s.t. }\; \text{d}\overline{X}_{t} = -\int_{\mathds{R}^{n}}\upnu_{t}(y) \frac{\overline{X}_{t} - y}{\gamma} \, \rho^\infty (\text{d}y;\overline X_{t})\,\text{d}t\\
        & \text{and }\; \overline{X}_{0} = x \in \mathds{R}^{n},\quad t\in [0,T] .
    \end{aligned}
\end{equation}
where $\rho^\infty$ is the Gibbs measure defined by \eqref{Gibbs}. 
Recall that if $\upnu_t(y) = u_t$ is constant in $y$,  then the dynamics of this system is $\overline f(x_{t}, \upnu_t)={ -}u_{t}\nabla \phi_{\gamma}(x_{t})$, where $\phi_{\gamma}$ is the local entropy associated to $\phi$, see \eqref{local entropy}. 
We will assume that\vspace*{-0.5em}
\begin{equation}
\label{assfi}
\phi\in C^{1}(\mathds{R}^{n}) \,, \nabla\phi\,\text{ is Lipschitz continuous with constant } L \,, \; 0<\gamma<\frac{1}{L} .
\end{equation}

\begin{cor}
\label{cor-conv-value}
If \eqref{assfi} holds, then the control system in \eqref{sp-learn} satisfies  (A') and (C), and $\rho^\infty(y,x)=\mu_x(y)$  satisfies (D), where  $\mu_x$ is the invariant probability measure of the fast subsystem \eqref{fast subsys} (see Prop. \ref{prop:innv}). Moreover, for any functional $J$ of the form \eqref{cost function} satisfying  (B'), the conclusions of Theorem \ref{thm-conv-value} and Theorem \ref{thm: value function sol limit pde} hold true. In particular, \vspace*{-0.5em}
\[
\lim_{\varepsilon\to 0} {\mathcal{V}^\varepsilon}(x,y) =    \overline{\mathcal{V}}(x) \,, \quad \text{locally uniformly}.
\]
\end{cor}

\begin{proof} 
The drift of the control system in \eqref{sp-learn} is given by
\[
f(x,y,u)= -u(x-y)/\gamma \,, \quad b(x,y)=-\nabla\phi(y) + (x-y)/\gamma .
\]
The conditions (A2), (A3), and (C) are trivial. For (A1), if $\bar u:= \max\{|u| : u\in U\}$,  $f$ is Lipschitz in $x, y$ with constant $\bar u/\gamma$ and satisfies \eqref{assumption-slow} with $C=\bar u/\gamma$. Similarly, $b$ is Lipschitz 
 with constant $L+ 1/\gamma$ and satisfies \eqref{assumption-fast} with $C=\nabla\phi(0)+ L+ 1/\gamma$.

To check (A4), we set $\kappa:=L-1/\gamma$, and  observe that\vspace*{-0.5em}
\[
\left(b(x,y)-b(x,z)\right)\cdot (y-z)= (\nabla\phi(z)- \nabla\phi(y))\cdot(y-z) - \frac{|y-z|^2}\gamma \leq -\kappa |y-z|^2.
\]
Under the current assumptions on $\phi$ the equality $\rho^\infty(y,x)=\mu_x(y)$ can be found, e.g., in \cite{bogachev2010invariant}, and assumption (D) is satisfied following Remark \ref{rem: D model}. Then the assumptions of Theorem \ref{thm-conv-value}  and Theorem \ref{thm: value function sol limit pde} are verified, because $g=\phi$ satisfies condition (B').
\end{proof}

\begin{proof}(of Theorem \ref{thm: practice}) 
The set of admissible controls $\mathcal{U}$ in the definition of the value function $\mathcal{V}$ is a subset of the extended control set $\mathcal{U}^{\text{ex}}$, since the latter contains all controls which are constant with respect to $y$ and $\mu_{x}$ is a probability measure (see Remark \ref{rmk: extended control set}). 
Hence by Corollary \ref{cor-conv-value} we have $\lim_{\varepsilon\to 0} \mathcal{V}^{\varepsilon}(x,y) = \overline{\mathcal{V}}(x) \leq \mathcal{V}(x)$.
\end{proof}

\begin{proof} (of Corollary \ref{cor:nocontrol}) 
By Corollary \ref{cor-conv-value} the assumptions of Theorem \ref{conv-thm-1} and Theorem \ref{conv-thm 2}  are verified. 
Since the control set $U$ is a singleton we immediately deduce from them the conclusions.
\end{proof}

The next result says that optimal trajectories in \eqref{eq proof: limit ocp} can be recovered as $\varepsilon\to 0$ by a sequence of controlled trajectories $X^{\varepsilon}$ that are quasi-optimal in \eqref{sp-learn}.

\begin{cor}
\label{cor:trajectories}
Assume \eqref{assfi}. Then,  for any 
controlled trajectory $\hat{x}$ of the system in problem \eqref{eq proof: limit ocp}, there exist a sequence of controlled trajectories $\big( X^{\varepsilon}, Y^{\varepsilon}\big)_{\varepsilon>0}$ of the system in \eqref{sp-learn} such that  \vspace*{-0.7em}
\begin{equation*}
    \lim\limits_{\varepsilon\to 0} \mathds{E}\left[|{X}^{\varepsilon}_{T} - \hat{x}_{T}|^{2}\right] + \int_{{0}}^{T}\mathds{E}\left[\big| X^{\varepsilon}_{s} - \hat{x}_{s}\big|^{2}\right]\,\text{d}s = 0. 
\end{equation*}
If, moreover, $\phi$ is bounded and $(\hat{x},\hat{\upnu})$ is optimal for \eqref{eq proof: limit ocp}, then this sequence is { quasi-optimal} in the sense that \vspace*{-0.3em}
\begin{equation*}
	\mathds{E}[\phi(X^{\varepsilon}_{T})] \leq \mathcal{V}^{\varepsilon}(x,y) + o(1), \quad \text{as } \varepsilon\to 0 \,.
\end{equation*} 
\end{cor}

\begin{proof}
By Corollary \ref{cor-conv-value} we can invoke
Theorem \ref{conv-thm-1}. The quasi-optimality of the approximating sequence is a consequence of Theorem \ref{thm-quasi-opt}.
\end{proof}

 \begin{proof} (of Corollary \ref{cor:trajectories (old)})
The assumptions of  Theorem \ref{conv-thm 2} are verified by Corollary \ref{cor-conv-value}. Then we know that $\dot{\overline{x}}_{s} \in \overline{\text{co}}\, \overline{f}(\overline{x}_{s}, U^{ex})$ for  a.e. $s\in [t,T]$, where 
\[
 \overline{f}(x,\upnu) = -\int_{\mathds{R}^{n}}\upnu(y) \frac{x - y}{\gamma} \,\text{d}\mu_x(y) ,
\]
which is linear in $\upnu\in U^{ex}$. Since $U$ is convex, also $U^{ex}$ is convex.
Then $\overline{\text{co}}\, \overline{f}(x, U^{ex})= \overline{f}(x, U^{ex})$ and so $\overline{x}_.$ is a trajectory of the effective system in \eqref{effective_sp-learn}. Therefore $\phi({\overline{x}}_T) \geq \overline{\mathcal{V}}(x)$. 
Next, recall from  the proof of Theorem \ref{thm: practice} that $\lim_{\varepsilon\to 0} \mathcal{V}^{\varepsilon}(x,y) = \overline{\mathcal{V}}(x)$. Then \eqref{subop} and \eqref{ineq: cor} imply $\phi({\overline{x}}_T) \leq \overline{\mathcal{V}}(x)$, which gives the conclusion.
\end{proof}

\section*{Acknowledgments}
We are very grateful to a referee for a precious feedback which helped us to improve considerably the paper. We also thank Alekos Cecchin for useful discussions about the dynamic programming principle and for pointing out to us the papers \cite{djete2022mckean} and \cite{karoui2013capacities}.

\bibliography{bibliography}

\begin{bibdiv}
\begin{biblist}

\bib{aubin2009set}{book}{
      author={Aubin, Jean-Pierre},
      author={Frankowska, H{\'e}l{\`e}ne},
       title={Set-valued analysis},
   publisher={Birkh\"auser, Boston},
        date={1990},
}

\bib{baldassi2015subdominant}{article}{
      author={Baldassi, Carlo},
      author={Ingrosso, Alessandro},
      author={Lucibello, Carlo},
      author={Saglietti, Luca},
      author={Zecchina, Riccardo},
       title={Subdominant dense clusters allow for simple learning and high computational performance in neural networks with discrete synapses},
        date={2015},
     journal={Physical review letters},
      volume={115},
      number={12},
       pages={128101},
}

\bib{bardi2011optimal}{article}{
      author={Bardi, Martino},
      author={Cesaroni, Annalisa},
       title={Optimal control with random parameters: a multiscale approach},
        date={2011},
     journal={European journal of control},
      volume={17},
      number={1},
       pages={30\ndash 45},
}

\bib{bardi2010convergence}{article}{
      author={Bardi, Martino},
      author={Cesaroni, Annalisa},
      author={Manca, Luigi},
       title={Convergence by viscosity methods in multiscale financial models with stochastic volatility},
        date={2010},
     journal={SIAM Journal on Financial Mathematics},
      volume={1},
      number={1},
       pages={230\ndash 265},
}

\bib{bardi2023singular}{article}{
      author={Bardi, Martino},
      author={Kouhkouh, Hicham},
       title={Singular perturbations in stochastic optimal control with unbounded data},
        date={2023},
     journal={ESAIM: Control, Optimisation and Calculus of Variations},
      volume={29},
       pages={52},
}

\bib{bhatia1997matrix}{book}{
      author={Bhatia, Rajendra},
       title={Matrix {A}nalysis},
   publisher={Springer-Verlag, New York},
        date={1997},
}

\bib{billingsley2008probability}{book}{
      author={Billingsley, Patrick},
       title={Probability and measure},
   publisher={John Wiley \& Sons},
        date={2008},
}

\bib{bogachev2010invariant}{inproceedings}{
      author={Bogachev, VI},
      author={Kirillov, AI},
      author={Shaposhnikov, SV},
       title={Invariant measures of diffusions with gradient drifts},
organization={Pleiades Publishing, Ltd.},
        date={2010},
   booktitle={Doklady mathematics},
      volume={82},
       pages={790\ndash 793},
}

\bib{bogachev2000generalization}{article}{
      author={Bogachev, Vladimir~I},
      author={R{\"o}ckner, Michael},
       title={A generalization of {H}asminskii’s theorem on existence of invariant measures for locally integrable drifts},
        date={2000},
     journal={Theory Probab. Appl},
      volume={45},
      number={3},
}

\bib{bogachev2014kantorovich}{article}{
      author={Bogachev, Vladimir~Igorevich},
      author={Kirillov, Andrei~Igorevich},
      author={Shaposhnikov, Stanislav~Valer'evich},
       title={The {K}antorovich and variation distances between invariant measures of diffusions and nonlinear stationary {F}okker-{P}lanck-{K}olmogorov equations},
        date={2014},
     journal={Mathematical Notes},
      volume={96},
       pages={855\ndash 863},
}

\bib{bogachev2002uniqueness}{article}{
      author={Bogachev, Vladimir~Igorevich},
      author={Rockner, M},
      author={Stannat, W},
       title={Uniqueness of solutions of elliptic equations and uniqueness of invariant measures of diffusions},
        date={2002},
     journal={Sbornik: Mathematics},
      volume={193},
      number={7},
       pages={945},
}

\bib{borkar2007averaging}{article}{
      author={Borkar, Vivek},
      author={Gaitsgory, Vladimir},
       title={Averaging of singularly perturbed controlled stochastic differential equations},
        date={2007},
     journal={Applied mathematics and optimization},
      volume={56},
      number={2},
       pages={169\ndash 209},
}

\bib{borkar2007singular}{article}{
      author={Borkar, Vivek~S},
      author={Gaitsgory, Vladimir},
       title={Singular perturbations in ergodic control of diffusions},
        date={2007},
     journal={SIAM journal on control and optimization},
      volume={46},
      number={5},
       pages={1562\ndash 1577},
}

\bib{chaudhari2019entropy}{article}{
      author={Chaudhari, Pratik},
      author={Choromanska, Anna},
      author={Soatto, Stefano},
      author={LeCun, Yann},
      author={Baldassi, Carlo},
      author={Borgs, Christian},
      author={Chayes, Jennifer},
      author={Sagun, Levent},
      author={Zecchina, Riccardo},
       title={Entropy-{SGD}: Biasing gradient descent into wide valleys},
        date={2019},
     journal={Journal of Statistical Mechanics: Theory and Experiment},
      volume={2019},
      number={12},
       pages={124018},
}

\bib{chaudhari2018deep}{article}{
      author={Chaudhari, Pratik},
      author={Oberman, Adam},
      author={Osher, Stanley},
      author={Soatto, Stefano},
      author={Carlier, Guillaume},
       title={Deep relaxation: partial differential equations for optimizing deep neural networks},
        date={2018},
     journal={Research in the Mathematical Sciences},
      volume={5},
      number={3},
       pages={1\ndash 30},
}

\bib{clarke2008nonsmooth}{book}{
      author={Clarke, Francis~H},
      author={Ledyaev, Yuri~S},
      author={Stern, Ronald~J},
      author={Wolenski, Peter~R},
       title={Nonsmooth analysis and control theory},
   publisher={Springer-Verlag, New York},
        date={1998},
}

\bib{da2006uniqueness}{article}{
      author={Da~Lio, Francesca},
      author={Ley, Olivier},
       title={Uniqueness results for second-order {B}ellman--{I}saacs equations under quadratic growth assumptions and applications},
        date={2006},
     journal={SIAM journal on control and optimization},
      volume={45},
      number={1},
       pages={74\ndash 106},
}

\bib{djete2022mckean}{article}{
      author={Djete, Mao~Fabrice},
      author={Possama{\"i}, Dylan},
      author={Tan, Xiaolu},
       title={Mc{K}ean--{V}lasov optimal control: the dynamic programming principle},
        date={2022},
     journal={The Annals of Probability},
      volume={50},
      number={2},
       pages={791\ndash 833},
}

\bib{karoui2013capacities}{article}{
      author={El~Karoui, Nicole},
      author={Tan, Xiaolu},
       title={Capacities, measurable selection and dynamic programming. {P}art {II}: application in stochastic control problems},
        date={2015},
     journal={arXiv preprint arXiv:1310.3364},
}

\bib{fleming2006controlled}{book}{
      author={Fleming, Wendell~H},
      author={Soner, Halil~Mete},
       title={Controlled {M}arkov processes and viscosity solutions},
   publisher={2nd edition. Springer, New York},
        date={2006},
}

\bib{fouque2011multiscale}{book}{
      author={Fouque, Jean-Pierre},
      author={Papanicolaou, George},
      author={Sircar, Ronnie},
      author={S{\o}lna, Knut},
       title={Multiscale stochastic volatility for equity, interest rate, and credit derivatives},
   publisher={Cambridge University Press},
        date={2011},
}

\bib{kokotovic1999singular}{book}{
      author={Kokotovi{\'c}, Petar},
      author={Khalil, Hassan~K},
      author={O'Reilly, John},
       title={Singular perturbation methods in control: analysis and design},
   publisher={Academic Press, London},
        date={1986},
}

\bib{kouhkouh22phd}{article}{
      author={Kouhkouh, Hicham},
       title={Some asymptotic problems for {H}amilton-{J}acobi-{B}ellman equations and applications to global optimization},
        date={2022},
        note={PhD thesis, University of Padova. Available online \url{https://hdl.handle.net/11577/3444759}},
}

\bib{kushner2012weak}{book}{
      author={Kushner, Harold},
       title={Weak convergence methods and singularly perturbed stochastic control and filtering problems},
   publisher={Birkhäuser, Boston},
        date={1990},
}

\bib{lecun2015deep}{article}{
      author={LeCun, Yann},
      author={Bengio, Yoshua},
      author={Hinton, Geoffrey},
       title={Deep learning},
        date={2015},
     journal={{N}ature},
      volume={521},
      number={7553},
       pages={436\ndash 444},
}

\bib{li2017stochastic}{inproceedings}{
      author={Li, Qianxiao},
      author={Tai, Cheng},
      author={E, Weinan},
       title={Stochastic modified equations and adaptive stochastic gradient algorithms},
organization={PMLR},
        date={2017},
   booktitle={International conference on machine learning},
       pages={2101\ndash 2110},
}

\bib{pardoux2001poisson}{article}{
      author={Pardoux, E},
      author={Veretennikov, A~Yu},
       title={On the {P}oisson equation and diffusion approximation. {I}},
        date={2001},
     journal={Annals of probability},
       pages={1061\ndash 1085},
}

\bib{pardoux2003poisson}{article}{
      author={Pardoux, E},
      author={Veretennikov, A~Yu},
       title={On {P}oisson equation and diffusion approximation 2},
        date={2003},
     journal={The Annals of Probability},
      volume={31},
      number={3},
       pages={1166\ndash 1192},
}

\bib{pardoux2005poisson}{article}{
      author={Pardoux, Etienne},
      author={Veretennikov, A~Yu},
       title={On the {P}oisson equation and diffusion approximation 3},
        date={2005},
     journal={The Annals of Probability},
      volume={33},
      number={3},
       pages={1111\ndash 1133},
}

\bib{pavon2022local}{article}{
      author={Pavon, Michele},
       title={On local entropy, stochastic control, and deep neural networks},
        date={2022},
     journal={IEEE Control Systems Letters},
      volume={7},
       pages={437\ndash 441},
}

\bib{pittorino2021entropic}{article}{
      author={Pittorino, Fabrizio},
      author={Lucibello, Carlo},
      author={Feinauer, Christoph},
      author={Perugini, Gabriele},
      author={Baldassi, Carlo},
      author={Demyanenko, Elizaveta},
      author={Zecchina, Riccardo},
       title={Entropic gradient descent algorithms and wide flat minima},
        date={2021},
     journal={Journal of Statistical Mechanics: Theory and Experiment},
      volume={2021},
      number={12},
       pages={124015},
}

\bib{stroock1997multidimensional}{book}{
      author={Stroock, Daniel~W},
      author={Varadhan, SR~Srinivasa},
       title={Multidimensional diffusion processes},
   publisher={Springer-Verlag, Berlin-New York},
        date={1979},
      volume={233},
}

\bib{wihler2009holder}{article}{
      author={Wihler, Thomas~P},
       title={On the {H}{\"o}lder continuity of matrix functions for normal matrices},
        date={2009},
     journal={J. Inequal. Pure Appl. Math.},
      volume={10},
      number={4, Article 91, 5 pp.},
}

\bib{yong1999stochastic}{book}{
      author={Yong, Jiongmin},
      author={Zhou, Xun~Yu},
       title={Stochastic controls: {H}amiltonian systems and {HJB} equations},
   publisher={Springer-Verlag, New York},
        date={1999},
}

\end{biblist}
\end{bibdiv}
\bibliographystyle{siam}

\end{document}